\documentclass[a4paper]{article}
\usepackage[english]{babel}
\usepackage[utf8]{inputenc}
\usepackage[T1]{fontenc}
\usepackage[usenames, dvipsnames]{xcolor}
\usepackage{tikz}
\usepackage{amsthm}
\usepackage{amsmath}
\usepackage{latexsym}
\usepackage{float}
\usepackage{stmaryrd}
\usepackage{fullpage}
\usepackage{multirow}
\usepackage{diagbox}
\usepackage{comment}
\usepackage{amssymb}

\everymath{\displaystyle}

\newtheorem{theorem}{Theorem}
\newtheorem{obs}[theorem]{Observation}
\newtheorem{lemma}[theorem]{Lemma}
\newtheorem{cor}[theorem]{Corollary}

\newtheorem{prop}[theorem]{Proposition}

\newtheorem{claim}{Claim}[theorem]

\newcommand{\outcomeP}{\mathcal{P}}
\newcommand{\outcomeN}{\mathcal{N}}
\newcommand{\grundy}{\mathcal{G}}

\newcommand{\sstar}[1]{S_{#1}}

\newcommand{\app}[3]{#1\tikz[baseline=-4]{\draw (0,0) node[below] {\footnotesize #2};\draw (0,0) node[circle,fill=black,minimum size=0,inner sep=1]{} -- (0.3,0) node[circle,fill=black, minimum size=0, inner sep=1] {};} P_{#3}} 
\newcommand{\csg}[1]{$CSG(#1)$} 

\DeclareMathOperator{\mex}{mex}

\tikzstyle{noeud}=[circle, fill=black, inner sep= 0, minimum size = 4]
\tikzstyle{gros_noeud}=[circle, fill=black, inner sep= 0, minimum size = 8]

\makeatletter
\renewcommand*{\@fnsymbol}[1]{\ensuremath{\ifcase#1 \or * \or \dagger
	 \else\@arabic{\numexpr#1-2\relax}\fi}}
\makeatother

\newcounter{savecntr}
\newcounter{restorecntr}

\title{Connected Subtraction Games on Subdivided Stars\thanks{This work has been supported by the ANR-14-CE25-0006 project of the French National Research Agency.}}

\date{}

\author{Antoine Dailly \thanks{Corresponding author} $^,$\setcounter{savecntr}{\value{footnote}}\thanks{Univ Lyon, Université Lyon 1, LIRIS UMR CNRS 5205, F-69621, Lyon, France.} \and
	Julien Moncel \thanks{LAAS-CNRS, Université de Toulouse, CNRS, Université Toulouse 1 Capitole - IUT Rodez, Toulouse, France.} $^,$\thanks{Fédération de Recherche Maths à Modeler, Institut Fourier, 100 rue des Maths, BP 74, 38402 Saint-Martin d'Hères Cedex, France.} \and
	Aline Parreau \setcounter{restorecntr}{\value{footnote}}%
	\setcounter{footnote}{\value{savecntr}}\footnotemark
	\setcounter{footnote}{\value{restorecntr}}}

\begin{document}
	
	\maketitle
	
	
	

		\begin{abstract}
			The present paper deals with connected subtraction games in graphs, which are generalization of take-away games. In a connected subtraction game, two players alternate removing a connected subgraph from a given connected game-graph, provided the resulting graph is connected, and provided the number of vertices of the removed subgraph belongs to a prescribed set of integers. We derive general periodicity results on such games, as well as specific results when played on subdivided stars.
		\end{abstract}
		
		\textbf{Keywords:}
			Combinatorial Games;
			Subtraction Games;
			Graphs
	
	\section{Connected subtraction games on graphs}
	
	\subsection{General description of the game}
	In this paper, we study \emph{connected subtraction games} on graphs, which are impartial combinatorial games where a player can remove a connected subgraph from a given connected game-graph, provided the moves lead to a new game-graph which remains connected, and provided the number of vertices that have been removed is legal, with respect to a given list of integers that characterizes the game. Such a game will be denoted a $CSG$ game.
	
	More precisely, let $L$ be a set of positive integers and $G$ a connected graph. The game \csg{L} on $G$ is a 2-player game where, starting from $G$, a player can remove a connected subgraph $H$ from the current graph, whenever the number of vertices of $H$ belongs to $L$ and the remaining graph is still connected. The first player unable to play loses the game. See Figure \ref{fig:example} for an example of a \csg{\{1,2,4\}} game.

	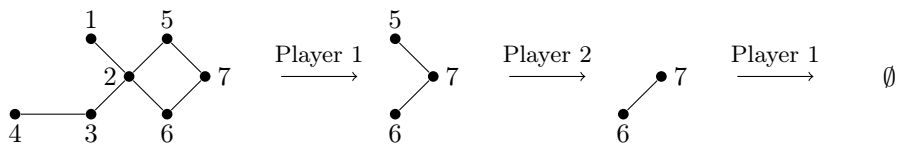
\begin{figure}[h]
		\begin{center}
			\begin{tikzpicture}
			\node[noeud] (1) at (-0.5,0.5) {};
			\draw (1) ++(0,0.25) node {1};
			\node[noeud] (2) at (0,0) {};
			\draw (2) ++(-0.25,0) node {2};
			\node[noeud] (3) at (-0.5,-0.5) {};
			\draw (3) ++(0,-0.25) node {3};
			\node[noeud] (4) at (-1.5,-0.5) {};
			\draw (4) ++(0,-0.25) node {4};
			\node[noeud] (5) at (0.5,0.5) {};
			\draw (5) ++(0,0.25) node {5};
			\node[noeud] (6) at (0.5,-0.5) {};
			\draw (6) ++(0,-0.25) node  {6};
			\node[noeud] (7) at (1,0) {};
			\draw (7) ++(0.25,0) node {7};
			\draw (1)--(2)--(3)--(4);
			\draw (2)--(5)--(7)--(6)--(2);
			
			\path[draw,->] (2,0) to node[above]{\small Player 1} (3,0) ;
			
			\begin{scope}[shift={(3,0)}]
			\node[noeud] (5) at (0.5,0.5) {};
			\draw (5) ++(0,0.25) node {5};
			\node[noeud] (6) at (0.5,-0.5) {};
			\draw (6) ++(0,-0.25) node  {6};
			\node[noeud] (7) at (1,0) {};
			\draw (7) ++(0.25,0) node {7};
			\draw (5)--(7)--(6);
			\end{scope}
			
			\path[draw,->] (5,0) to node[above]{\small Player 2} (6,0) ;
			
			\begin{scope}[shift={(6,0)}]
			\node[noeud] (6) at (0.5,-0.5) {};
			\draw (6) ++(0,-0.25) node  {6};
			\node[noeud] (7) at (1,0) {};
			\draw (7) ++(0.25,0) node {7};
			\draw (7)--(6);
			\end{scope}
			
			\path[draw,->] (8,0) to node[above]{\small Player 1} (9,0) ;
			
			\begin{scope}[shift={(10,0)}]
			\node at (0,0) {$\emptyset$};
			\end{scope}
			
			\end{tikzpicture}
			
		\end{center}
		\caption{An example of a \csg{\{1,2,4\}} game. The first player starts by taking the connected subgraph induced by vertices $\{1,2,3,4\}$. Then the second player answers by taking vertex $5$. Finally, the first player wins by taking the final edge $\{6,7\}$.}
		\label{fig:example}
	\end{figure}

	In some sense, this game is a natural generalization of popular take-away games, that are usually played with heaps of counters, in which a player can remove a number of counters belonging to a given set $L$. In the theory of combinatorial games, this is called a (simple) subtraction game (see, for instance, Chapter 4 in the first volume of \cite{winningways}). From a graph theory point of view, such subtraction games can be seen as games \csg{L} played on paths.
	
	Connected subtraction games on graphs have been addressed in \cite{033}, where \csg{\{1,2\}} is solved for subdivided stars and bistars. In the present paper we extend these results, by providing, on the first hand, insights for such games in general graphs, and, on the other hand, specific results for \csg{\{1,2,3\}} and \csg{\{1,2,4\}} in subdivided stars.
	
	\subsection{Notations for combinatorial games, outcomes and Grundy values}
	
	In this section, we recall basics of combinatorial games that we will need along the paper. For more about combinatorial games, the reader can refer to \cite{LIP, winningways, Siegel}.
	In the classical settings of combinatorial game theory, two players play in turn, without being able to pass. The game is finite (the number of positions is finite and we can not cycle through some positions), there are no random events, and the information is perfect (each player knows what is the situation and what are the possible moves). In addition, both players are identical, in the sense that they have the same available moves (such a game is called {\em impartial}). The first player having no more available moves loses the game (and the other player wins). Note that such games admit no draws.
	
	In this setting, a given game is either a winning position or a losing position for the first player. We can recursively define losing and winning positions as follows. A given game is a losing position if it is either a game in which there are no available moves, or if it is a game in which each possible move leads to a winning position. A given game is a winning position if there exists a move leading to a losing position. A losing position is denoted as a $\outcomeP$-position, whereas a winning one is denoted as a $\outcomeN$-position.
	
	We can also define recursively the Grundy value of a given game $G$, denoted $\grundy(G)$, as 0 if there are no possible moves, and as $\mex\{\grundy(G_1),\ldots,\grundy(G_k)\}$ otherwise, where $G_1,\ldots,G_k$ denote all the options of $G$ that are all the games we may obtain after playing one move on $G$. Recall that the $\mex$ of a finite sequence of nonnegative integers is defined as the smallest nonnegative integer not belonging to the sequence. Note that it follows from the definition that $G$ is a losing position if, and only if, $\grundy(G)=0$. Winning positions correspond to games $G$ for which we have  $\grundy(G)>0$. Indeed, a game has Grundy value $g>0$ if, and only if, no move leads to a game having Grundy value $g$, and for any $0\leq g'<g$, there exists a move leading to a game having Grundy value $g'$.
	Therefore, the Grundy value is a refinement of the notion of outcome, i.e winning or losing positions. We generally consider that a game is completely solved if we know its Grundy value.
	
	Grundy values are particularly useful when dealing with games that can be decomposed as sum of games. The sum of two games $G_1$ and $G_2$, denoted $G_1+G_2$, is the new game where a player can either play on $G_1$ (with any legal move on that game) or play on $G_2$ (with any legal move on that game).
	
	It turns out that the Grundy value of a sum of games can be rather easily computed from the Grundy values of the games. It is easy to see that, for instance, $\grundy(G+G)=0$ for any game $G$. Indeed, we can show recursively that any move the first player makes on $G+G$ can be mimicked by the second player so as to leave a game $G'+G'$. Since there are no moves in $\emptyset+\emptyset$, this shows that $\grundy(G+G)=0$. More generally, the so-called Sprague-Grundy theorem \cite{Sprague} states that the binary value of $\grundy(G_1+G_2)$ can be computed as $\grundy(G_1)\oplus\grundy(G_2)$, where $\oplus$ denotes the bitwise XOR of the binary values of both games. For instance, the Grundy value of $G_1+G_2$ is 3 if $\grundy(G_1)=5$ and $\grundy(G_2)=6$, since $5_2=101$, $6_2=110$, and $101\oplus 110=011$, that is to say 3. A consequence of the Sprague-Grundy theorem is that two games $G_1$ and $G_2$ have the same Grundy value if and only if $G_1+G_2$ is a losing position. In this case we say that $G_1$ and $G_2$ are equivalent, denoted by $G_1 \equiv G_2$. We will often use this result to compute the Grundy value of a game.
	
	In this paper, we will denote by $\grundy_L(G)$ the Grundy value of the game \csg{L} played on a graph $G$. If the context is clear, we will simply write $\grundy(G)$.
	
	\subsection{Literature review}
	
	It is well-known (see for instance \cite{winningways}) that any simple subtraction game with a finite set $L$ is ultimately periodic, in the sense that there exists a period $T>0$ and an integer $k_0$ such that, for any $k\geq k_0$, we have $\grundy(G_{k+T})=\grundy(G_{k})$, where $G_k$ denotes the game played with a heap of $k$ counters. Such ultimately periodic sequences of integers can be described using the $\overline{x_1,\ldots,x_k}$ notation, where $\overline{x_1,\ldots,x_k}$ denotes the infinite sequence $x_1,\ldots,x_k,x_1,\ldots,x_k,\ldots$. For instance, it is well-known and easy to check that $(\grundy(G_k))_{k\in \mathbb N}=\overline{0123}$ for the subtraction game with $L=\{1,2,3\}$, meaning that, for instance, $\grundy(G_2)=2$ and $\grundy(G_4)=0$ (the sequence starts with the value corresponding to $k=0)$. Note that, in general, the sequence is ultimately periodic, meaning we may have a preperiod \cite[Chapter 4]{winningways}. For instance, for the subtraction game with $L=\{2,4,7\}$, the Grundy sequence has preperiod 8 and period $3$, since we have:
	
	$$(\grundy(G_k))_{k\in\mathbb N}=00112203\overline{102}.$$
	
	Subtraction games are take-away games\footnote{Also called Nim-like games, or octal games.}, which is the family of games that includes more or less all combinatorial games plays on heaps of counters where players remove counters with some specific rules. The ultimate periodicity of such games has been conjectured  by Guy \cite{Guy96} and is a key problem in Combinatorial Game Theory. See \cite{winningways} for more details on take-away games and their periodicity.
	
	There exist several generalizations of take-away games in graphs. In \cite{Nim-graphs, Nim-graphs2}, the edges of the graphs are labeled with an integer and a token is moved along the edges. In Node-Kayles \cite{Node-Kayles}, a move consists in removing a vertex and all its neighbours from the graph. 
	If instead of removing vertices players remove edges, we obtain Arc-Kayles \cite{Node-Kayles}.
	In the game Grim defined and studied in \cite{Grim}, a move consists in removing a vertex, deleting its adjacent edges, and deleting the resulting isolated vertices. Finally, a general definition of octal games on graphs and specifically of subtraction games has been recently given in \cite{033}. 
	
	A natural question that arises in these generalizations in graphs is whether the (ultimate) periodicity of the Grundy sequences that appear in classical take-away games (that can be considered as a particular case of connected subtraction games played on paths) is still valid for more complicated graphs. This question has been studied for Node-Kayles \cite{fleischer} and Arc-Kayles \cite{Huggan} in the particular class of subdivided stars (which is the simplest generalization of paths). For both games, they are able to prove that the sequence is still (ultimately) periodic for specific subdivided stars with three rays. Huggan and Stevens \cite{Huggan} have conjectured that this is true for Arc-Kayles for all subdivided stars with three rays, one of which is of size~1, whereas Fleischer and Trippen \cite{fleischer} proved that this is not true for Node-Kayles.
	As for the game \csg{\{1,2\}} on graphs, it is proved in \cite{033} that it is periodic in all subdivided stars and bistars.
	More precisely, if we denote by $\sstar{\ell_1,\ldots,\ell_k}$ the subdivided star obtained by appending to a single vertex $k$ finite paths of length $\ell_1,\ldots,\ell_k$, it is shown in \cite{033} that
	$$\grundy_{1,2}(\sstar{\ell_1,\ldots,\ell_k}) = \grundy_{1,2}(\sstar{\ell_1 \bmod 3,\ldots,\ell_k \bmod 3})$$ for all $\ell_1,\ldots,\ell_k$. In other words, adding three vertices to any path of a subdivided star does not change its Grundy value. Note that this value 3 is also equal to the period of the subtraction game with $L=\{1,2\}$. Similar results for subdivided bistars are also derived in \cite{033}, where a bistar is a graph obtained by connecting with a path the centers of two subdivided stars. 
	
	\subsection{Outline of the paper}
	In this paper we are interested in deriving general results on $CSG$ games, with a particular emphasis on the correlation of the structure of the graphs with the structure of the Grundy values.
	
	Most of our results are about the behaviour of the Grundy value when one appends a path of varying size to a vertex of a graph. Let $G$ be a graph, $u$ one of its vertices and $k$ a positive integer. We denote by $\app{G}{u}{k}$ the graph $G$ where a path $P_k$ of $k$ vertices is appended to $u$ ($u$ not being a part of this path).
	We extend this notation to the case where $G$ is empty, by setting in this case $\app{G}{u}{k}=P_k$. We are interested by the periodicity of the function $$\begin{array}{ccccc}
	f_{L,G,u} & : & \mathbb{N} & \to & \mathbb{N} \\
	& & k & \mapsto & \grundy_L{(\app{G}{u}{k})}. \\
	\end{array}$$ for various sets $L$.

	In Section~\ref{sec:generalperiod}, we give results for any subtraction game on a finite set. In particular, we extend the result of ultimate periodicity on paths to any graph. In other words, the function $f_{L,G,u}$ is ultimately periodic for any finite set $L$, graph $G$ and vertex $u$ of $G$. 
	In the rest of the paper, we aim to prove periodicity without preperiod.
	In Section~\ref{sec:123}, we derive general results for the game \csg{\{1,2,\ldots,N\}}, that is the game where any connected subgraph of size up to $N$ can be removed. These results consider general graphs and subdivided stars to which paths are appended. They allow us to solve the particular case \csg{\{1,2,3\}} in subdivided stars.
	In Section~\ref{sec:124}, we derive periodicity results for \csg{\{1,2,4\}} in subdivided stars.

	\section{General results on finite subtraction games}\label{sec:generalperiod}

	As said in the introduction, simple subtraction games with a finite set $L$ have ultimately periodic Grundy sequences (see Chapter 4 of \cite{Siegel} for a proof). Using our terminology, this means that if $L$ is finite, the function $f_{L,G,u}$ is ultimately periodic when $G$ is a path and $u$ is an endpoint of $G$. The main argument to prove this result is that when a path is long enough, all the moves in $L$ are possible and the number of moves is finite. Thus the Grundy value of a long path is the mex over a finite set. This set can only have a finite number of values which proves the periodicity.  We extend this argument to any graph and any vertex.
	
	\begin{theorem}\label{thm:general}
		Let $L$ be a finite set of positive integers, $G$ a graph and $u$ a vertex of $G$. Then the function $f_{G,L,u}$ is ultimately periodic. 
	\end{theorem}
	
	\begin{proof}
		We proceed by induction on $|G|$, where $|G|$ denotes the number of vertices of $G$. The case $|G|=1$ corresponds to the path and to simple subtraction games.
		
		Let $G$ be a graph with at least two vertices and $u$ be a vertex of $G$. Let $G'\neq G$ be a connected subgraph of $G$ containing $u$. Assume that the function $f_{G',L,u}$ is ultimately periodic with period $T(G')$ and preperiod $k(G')$. Let $k_0$ and $T_0$ be the preperiod and the ultimate period of the Grundy sequence of the subtraction game with set $L$. Let $T$ be the lcm of $T_0$ and all the values $T(G')$, with $G'\neq G$ any connected subgraph of $G$ containing $u$, and let $k_{max}$ be the maximum of the preperiods among the preperiod $k(G')$ and $k_0$.
		
		When $k\geq \max{L}$, then the set of available moves is fixed, and there are three kinds of moves from the position $\app{G}{u}{k}$:
		\begin{enumerate}
			\item playing to $\app{G}{u}{k-i}$ with $i\geq1$ (at most $|L|$ moves);
			\item playing to $\app{G'}{u}{k}$, with $G'\neq G$ a connected subgraph of $G$ containing $u$ (at most $2^{|G|-1}$ moves);
			\item playing to $P_{k-i}$ with $i\geq 1$ (at most $|L|$ moves).
		\end{enumerate}
		
		Thus, the total number of moves is bounded by a constant $M$ only depending of $L$ and $G$ (but not depending of $k$). Thus the Grundy value of $\app{G}{u}{k}$, $f_{G,L,u}(k)$, is bounded by $M$.
		
		Let $k\in \mathbb N$ and let $A(k)\in \{0,\ldots,M\}^{|L|}$ be the $|L|$ consecutive values of the Grundy values of $\app{G}{u}{i}$ from $i=k+1$ to $i=k+|L|$: $$A(k)=(f_{G,L,u}(k+1),f_{G,L,u}(k+1),...,f_{G,L,u}(k+|L|)).$$
		
		The vector $A(k)$ is an $|L|$-uplet with values in $\{0,...,M\}$. Since there are a finite number of such $|L|$-uplets, and since the number of such vectors $A(k)$ is infinite, then there exist $k_1$ and $k_2$ such that $A(k_1)=A(k_2)$. Moreover, we can assume that $k_{max}+|L|\leq k_1 \leq k_2$ and that $k_1\equiv k_2 \equiv 0 \bmod{T}$. Indeed, there are also infinitely many vectors $A(k)$ with $k\geq k_{max}+|T|$ and $k\equiv0\bmod{T}$, hence at least two of them must be equal.
		
		Let $T_f=k_2-k_1$. We prove by induction that for any $k> k_1$, $f_{G,L,u}(k+T_f)=f_{G,L,u}(k)$ which will conclude the proof. By definition of $k_1$ and $k_2$, the result is true for $k_1< k \leq k_1+|L|$. Let $k>k_1+|L|$ and assume that the result is true for the $|L|$ values preceding $k$. We prove that this is still true for $k$.
		Remembering that, since $k \geq \max L$, then the set of available moves is fixed, and the value $f_{G,L,u}(k+T_f)$ is the mex of the following values:
		\begin{enumerate}
			\item $f_{G,L,u}(k+T_f-i)$ with $\leq i\leq |L|$;
			\item $f_{G',L,u}(k+T_f)$, with $G'\neq G$ a connected subgraph of $G$ containing $u$;
			\item $\grundy(P_{k+T_f-i})$ with $1\leq i\leq |L|$.
		\end{enumerate}
		
		By induction hypothesis, $f_{G,L,u}(k+T_f-i)=f_{G,L,u}(k-i)$.
		Furthermore, $T_f$ is a multiple of $T$ and $k \geq k_{max}$, and thus $f_{G',L,u}(k+T_f)=f_{G',L,u}(k)$ and $\grundy(P_{k+T_f-i})=\grundy(P_{k-i})$.
		Finally, since the set of available moves from $\app{G}{u}{k+T_f}$ and from $\app{G}{u}{k}$ are identical, then $f_{G,L,u}(k+T_f)$ and $f_{G,L,u}(k)$ are actually both a mex computed on the same set of values, and thus are equal.
	\end{proof}
	
	The preperiod and the period obtained in Theorem \ref{thm:general} can be arbitrary large, even compared to the period of the Grundy sequence of the subtraction game played on a path. However, in all the particular cases we have considered, the period is the same than for the path. We wonder if this is true for all subtraction games. Concerning the preperiod, for some simple subtraction games there is no preperiod and the Grundy sequence is purely periodic. This is the case for example when the set $L$ contains all the integers from $1$ to $N$. In the rest of the paper, we aim at proving "pure periodicity" theorems.
	
	
	Some of our results are using the following lemma.
	
	\begin{lemma}\label{lem:general}
		Let us consider the game $CSG(L)$, with $L$ a finite set, played on some family of graphs $\mathcal{F}$ which has the property that the Grundy value of the game when played on any graph $G$ of $\mathcal{F}$ depends only on the cardinality of $G$ modulo a given period. More precisely, let us assume that there exists $T>0$ and integers $\alpha_0,\ldots,\alpha_{T-1}$ satisfying $\{\alpha_0,\ldots,\alpha_{T-1}\} = \{0,\ldots,T-1\}$ such that $\grundy(G)=\alpha_{\Vert G\Vert}$ for any $G \in \mathcal{F}$, with $\Vert G\Vert$ defined as $|G| \bmod T$.
		Let $\mathcal{F}'$ be a family of graphs such that:
		\begin{itemize}
			\item any legal move played on a graph $G \in \mathcal{F}'$ leads to a graph of $\mathcal{F}$
			\item the set of legal moves played on a graph $G \in \mathcal{F}'$ satisfies
			\begin{equation}\label{eq:legal_moves}
			\{k \bmod T \mid \mbox{there exists a legal move removing } k \mbox{ vertices to } G\} = \{1,\ldots,T-1\}
			\end{equation}
		\end{itemize}
		Then, for any graph $G \in \mathcal{F}'$, we also have $\grundy(G)=\alpha_{\Vert G\Vert}$.
	\end{lemma}
	
	\begin{proof}
		Let us denote $H^k$ the set of graphs of $\mathcal{F}$ that are obtained after playing any legal move consisting in removing $k$ vertices to a graph $G$ of $\mathcal{F}'$. By definition of $\mathcal{F}$, we have $\grundy(H) = \alpha_{\Vert H \Vert}$ for any $H \in H^k$. Since $|H| = |G|-k$, then we have $\grundy(H) = \alpha_{\Vert G \Vert-k \bmod T}$  for any such graph. Now, since~(\ref{eq:legal_moves}) is assumed to be true, we have
		\begin{eqnarray*}
			\grundy(G) &=& \mex\{\grundy(H)\mid H\mbox{ is obtained by a legal move from }G\}\\
			&=& \mex\{\alpha_{\Vert G \Vert-k\bmod T} \mid \mbox{there exists a legal move removing }k\mbox{ vertices to }G\}\\
			&=&\mex\{\alpha_{\Vert G \Vert-1\bmod T},\alpha_{\Vert G \Vert-2\bmod T},\ldots,\alpha_{\Vert G \Vert-T+1\bmod T}\}
		\end{eqnarray*}
		Since we assumed moreover that $\{\alpha_0,\ldots,\alpha_{T-1}\} = \{0,\ldots,T-1\}$, we conclude that $\grundy(G)=\alpha_{\Vert G \Vert}$, which is the desired result.
	\end{proof}
	
	By inductively using this lemma, we will derive periodicity results for some specific families of subdivided stars, that will be useful for the study of $CSG(\{1,3\})$ and $CSG(\{1,2,4\})$ in subdivided stars.
	
	In the sequel, this lemma will mainly be used in the special case where $\alpha_i=i$ for all $i$. In other words, $\mathcal{F}$ is such that $\grundy(G)=|G| \bmod T$ for all $G \in \mathcal{F}$.

	We end this section with a simple observation.
	
	\begin{obs}\label{obs:|V|_bound}
		For any graph $G$ and any set $L$ of integers, we have $\grundy_{L}(G) \leq |G|$.
	\end{obs}
	
	\begin{proof}
		We prove the result by induction on $G$. This is true if $G$ is empty since the Grundy value of the empty graph is 0.
		Let $G'$ be a graph obtained after a legal move on $G$. $G'$ has strictly less vertices than $G$. Thus by induction, $\grundy_L(G')\leq |G'|<|G|$. Finally, since $\grundy_L(G)=\mex\{\grundy_L(G') | G' \text{ option of } G\}$, we have $\grundy_L(G)\leq |G|$.
	\end{proof}
	
	\section{Connected subtraction games where we can remove up to $N$ vertices}\label{sec:123}
	
	For a general finite set $L$, the period and the preperiod of the Grundy sequence of Theorem \ref{thm:general} can be very large, as it is already the case for paths. 
	Since the subtraction game on paths for the set $L=\{1,...,N\}=I_N$ is very simple, we focus on this game in this section and give periodicity results with small period, and without --- or with a very short --- preperiod. We first give general results, valid for any $N$, and any graph $G$. Then we focus on two families of subdivided stars and prove results that are still valid for any $N$. Finally, we apply these results to $N=3$ and completely solve it for subdivided stars. Note that the case $N=2$ is solved in \cite{033}.

	\subsection{Results on general graphs}
	
	We first recall the Grundy sequence of the game \csg{I_N} on paths:

	\begin{lemma}[{\em Folklore}]\label{lem:In-paths}
		Let $N$ and $k$ be two integers. We have $\grundy_{I_N}(P_k) = k \bmod (N+1)$, i.e. $(\grundy_{I_N}(P_k))_{k\in\mathbb N} = \overline{012\ldots N}$, where $P_k$ denotes the path on $k$ vertices.
	\end{lemma}

	We now give some results on $\grundy_{I_N}(G)$ that are just depending on the number of vertices of $G$, when this last one is small enough. We will generalize this result when considering some specific stars (see Lemma \ref{lem:01s1kl}).

	\begin{lemma}\label{lem:sizeOfGraph}
		For any graph $G$ and any integer $N\geq 1$, we have $\grundy_{I_N}(G) = |G|$ if $|G|\in \{0,1\}$, $\grundy_{I_N}(G) \geq 2$ if $|G|\in \{2,\ldots,N\}$, $\grundy_{I_N}(G) = 0$ if $|G|=N+1$, and $\grundy_{I_N}(G) = 1$ if $|G|=N+2$.
	\end{lemma}
	
	\begin{proof}
		There is no possible move in the empty graph and since $1\in I_N$, the only possible move from a single vertex is to take it to play to the empty graph. Thus it is clear that $\grundy_{I_N}(G)=|G|$ for $G$ with at most one vertex.
		
		Assume now that $|G|\in \{2,\ldots,N\}$. There exists a vertex $u$ of $G$ such that $G-u$ is connected (take, for instance, $u$ as an endpoint of a diameter of $G$). Then removing the whole graph or removing $G-u$ are legal moves leading to a graph with at most one vertex. Thus $\grundy_{I_N}(G)\geq 2$.
		
		Now, if $G$ has $N+1$ vertices, then every legal move leads to a graph $G'$ having between $1$ and $N$ vertices, and thus having a positive Grundy value. Hence $\grundy_{I_N}(G)=0$.
		
		Similarly, if $G$ has $N+2$ vertices, one can play to a graph $G'$ with $N+1$ vertices since there exists a vertex not disconnecting the graph, thus to a graph with Grundy value $0$. But there is no possible move to a graph of Grundy value $1$. Finally $\grundy_{I_N}(G)=1$.
	\end{proof}
	
	This lemma can be extended when one appends a path of size $N+1$ to $G$.
	\begin{lemma}\label{lem:general_bounds}
		Let $G$ be a graph, $u$ a vertex of $G$ and $G'=\app{G}{u}{N+1}$.
		We have $\grundy_{I_N}(G') = |G|$ if $|G|\in \{0,1\}$, $\grundy_{I_N}(G') \geq 2$ if $|G|\in \{2,\ldots,N\}$, $\grundy_{I_N}(G') = 0$ if $|G|=N+1$, and $\grundy_{I_N}(G') = 1$ if $|G|=N+2$.
	\end{lemma}
	
	\begin{proof}
		Let $n$ be the number of vertices of $G$.
		If $n\leq 1$, then $G'$ is a path and the result is true.
		
		If $n \in \{2,\ldots,N\}$, then one can play from $G'$ to $P_{N+1}$ by removing $G$ and to $\app{G}{u}{N+2-n}$ which has $N+2$ vertices. By Lemma \ref{lem:sizeOfGraph}, these two graphs have Grundy values 0 and 1, thus $\grundy_{I_N}(G')\geq 2$.
		
		Assume now that $n=N+1$. Let $H'$ be an option of $G'$. Either $H'=\app{G}{u}{i}$ with $i\in \{1,...,N\}$ or $H'=\app{H}{u}{i}$ where $H$ is an option of $G$ not containing $u$ (taking $u$ but not the whole graph $G$ would disconnect $G'$).
		In the first case, one can play from $H'$ to $G$ which has Grundy value $O$. Thus $\grundy_{I_N}(H')\neq 0$. In the second case, as seen before, since $H$ has between $1$ and $N$ vertices, we also have $\grundy_{I_N}(H')\neq 0$. Finally, $\grundy_{I_N}(G')=0$.
		
		Finally, assume that $n=N+2$. As before, the options $H'$ of $G'$ are either of the form $H'=\app{G}{u}{i}$ with $i\in \{1,...,N\}$ or $H'=\app{H}{u}{i}$ where $H$ is a option of $G$ not containing $u$. In particular, one can take for $H$ in the second case the graph $G$ where a vertex not disconnecting $G$ and different from $u$ is removed. This option has Grundy value $0$. Thus $\grundy_{I_N}(G')>0$. In the first case, one can play to $G$ with Grundy value $1$ (by the previous lemma) and thus $\grundy_{I_N}(H')\neq 1$. In the second case, since $H$ has between $2$ and $N+1$ vertices, we also have $\grundy_{I_N}(H')\neq 1$. Finally, $\grundy_{I_N}(G')=1$.
	\end{proof}

	In order to derive periodicity results, we will often prove that $G$ and $\app{G}{u}{N+1}$ have the same Grundy value. To prove such a result, a classical way is to prove the equivalent statement that $G+\app{G}{u}{N+1}$ is a $\outcomeP$-position. We will generally prove this by proving that any move of the first player can be answered by the second player to a move to a $\outcomeP$-position. The following proposition shows that, somehow, the only case to be addressed is the one where the first player removes vertices in $G$, and in particular removes the vertex $u$.
	
	\begin{prop}\label{prop:winning_move}
		Let $N$ be an integer. Let $G$ be a graph and $u$ be a vertex of $G$. Assume $G$ is minimal in the sense that $\grundy_{I_N}(G)\neq\grundy_{I_N}(\app{G}{u}{N+1})$ but $\grundy_{I_N}(G')=\grundy_{I_N}(\app{G'}{u}{N+1})$ for any connected subgraph $G'$ of $G$ containing $u$. Since $\grundy_{I_N}(G+\app{G}{u}{N+1})\neq 0$, then there exists a winning move on $G+\app{G}{u}{N+1}$. Then every winning move of $G+\app{G}{u}{N+1}$ is in the component $G$, removes the vertex $u$ and leaves at least two vertices.
	\end{prop}
	
	\begin{proof}
		Any move in $P_{N+1}$ can be completed to leave $G + G$ to the first player, which is a $\outcomeP$-position. Any move in $\app{G}{u}{N+1}$ leading to $\app{G'}{u}{N+1}$ with $G'$ a connected subgraph of $G$ containing $u$ can be replicated so as to leave $G'+\app{G'}{u}{N+1}$, which is a $\outcomeP$-position since $G$ is assumed to be minimal. Now let us consider a move in $\app{G}{u}{N+1}$ leading to $P_{N+1-i}$. This implies that $|G|\leq N-i$. By Observation~\ref{obs:|V|_bound}, this implies $g=\grundy_{I_N}(G) \leq N-i < N+1-i$, hence there exists a move from $P_{N+1-i}$ to $P_g$. In the whole game, the first player gets $G+P_g$ which has Grundy value $0$. Thus, if a winning move on $G+\app{G}{u}{N+1}$ exists, then this move is on $G$, and leads to a connected subgraph $G'$ of $G$. Moreover, if $u$ is still in $G'$ or if the move is taking the whole graph $G$, then this move can be replicated by the second player, so as to leave $G' +\app{G'}{u}{N+1}$ or $P_{N+1}$, which is in both cases a $\outcomeP$-position. Finally, if $G'$ is a single vertex $v$, then the second player can take at the end of $P_{N+1}$ the same number of vertices, that is $|G|-1$, leaving a graph with $N+2$ vertices. Thus at the end, using Lemma \ref{lem:sizeOfGraph}, the two components of the sum have Grundy value 1 and thus the sum is a $\outcomeP$-position.
	\end{proof}
	

	
	
	
	\subsection{Simple stars}
	
	In this subsection, we consider the case when $G$ is a simple star to which a path is appended. We call a simple star a graph with a central vertex connected to some vertices of degree 1, that we call leaves. We will denote by $\sstar{1^t}$ the simple star with $t$ leaves. We assume that $N\geq 3$ since the case $N=2$ is solved in \cite{033}.
	
	The periodicity of the Grundy values is an easy consequence of Proposition \ref{prop:winning_move}.
	
	\begin{prop}\label{prop:Stk_general}
		Let $t$ be a positive integer, $u$ the central vertex of $\sstar{1^t}$ and $N\geq 3$. For any $k\geq0$, we have $\grundy_{I_N}(\app{\sstar{1^t}}{u}{k})=\grundy_{I_N}(\app{\sstar{1^t}}{u}{k \bmod N+1})$. In particular, the function $f_{I_N,\sstar{1^t},u}$ is purely periodic with period $N+1$.
	\end{prop}

	\begin{proof}
		
		Let $G=\app{\sstar{1^t}}{u}{k}$ and $u'$ be the endpoint of $P_k$ in $G$ ($u'=u$ if $k=0$). 
		Then $\app{\sstar{1^t}}{u}{k+N+1}$ can be written $\app{G}{u'}{N+1}$. We need to prove that $G+\app{G}{u'}{N+1}$ is a $\outcomeP$-position for any $k$ and $t$.
		
		Assume that this is not true and let $(k,t)$ be the smallest integers (in lexicographic order) such that
		$G+\app{G}{u'}{N+1}$ is $\outcomeN$-position.
		
		Let $G'$ be a connected subgraph of $G$ containing $u'$. Either $G'$ is a path with endpoint $u'$, or $G'=\app{\sstar{1^{t'}}}{u}{k}$ with $t'<t$. 
		By minimality of $(k,t)$, and since the Grundy values of paths are periodic, then $G'+\app{G'}{u'}{N+1}$ is a $\outcomeP$-position.

		Thus we can apply Proposition \ref{prop:winning_move} on $G$, and the only possible winning move in $G+\app{G}{u'}{N+1}$ is a move in the component $G$ that is taking $u'$ and leaving at most two vertices.
		This corresponds to a move from $G$ to $G'=\app{\sstar{1^{t}}}{u}{k'}$ with $k'<k$. Then the second player can play in the component $\app{G}{u'}{N+1}$ to $\app{\sstar{1^t}}{u}{k'+N+1}$ by taking $k-k'$ vertices on the path.
		By minimality of $k$, the resulting position is a $\outcomeP$-position and thus the move from $G$ to $G'$ was not a winning move, a contradiction.
		
	\end{proof}

	Since the function is purely periodic, if one wants to know the Grundy value of $\app{\sstar{1^t}}{u}{k}$ we just need to compute the values for $k\leq N$.
	The values for $k=0$, that are actually the Grundy values of $\sstar{1^t}$ are given in the following lemma.

	\begin{lemma}
		\label{lem:csg1nstarsareP}
		Let $t\geq 1$ be a positive integer and $N\geq 3$. We have:
		$$\grundy_{I_N}(\sstar{1^t})=\left\{
		\begin{array}{ll}
		2 & \mbox{if } t \leq N-1 \mbox{ and } t \mbox{ odd} \\
		3 & \mbox{if } t \leq N-1 \mbox{ and } t \mbox{ even} \\
		0 & \mbox{if } t = N+i \mbox{ with } i \mbox{ even} \\
		1 & \mbox{if } t = N+i \mbox{ with } i \mbox{ odd.}
		\end{array}
		\right.$$
	\end{lemma}
	
	\begin{proof}
		We use induction on $t$. The base cases are $t=1$ and $t=2$ for which $\sstar{1^t}$ is simply $P_2$ and $P_3$ respectively and the result holds.
		
		Assume now that $t\geq 3$. There are three cases:
		
		\noindent\textbf{Case 1:} $t$ is odd, $3 \leq t \leq N-1$
		
		We prove that $\sstar{1^t}+P_2$ is a $\outcomeP$-position. We examine every option of the first player:
		\begin{itemize}
			\item If the first player plays to either $\sstar{1^t} + P_1$ or $P_1 + P_2$, then the second player answers by playing to $P_1 + P_1$ which is a $\outcomeP$-position;
			\item If the first player plays to either $P_2$ or $\sstar{1^t}$, then the second player answers by playing by emptying the remaining graph;
			\item If the first player takes one leaf from $\sstar{1^t}$, then the second player answers by taking another leaf from $\sstar{1^t}$, leaving $\sstar{1^{t-2}} + P_2$. By induction hypothesis, this is a $\outcomeP$-position.
		\end{itemize}
		
		\noindent\textbf{Case 2:} $t$ is even, $4 \leq t \leq N-1$
		
		We prove that $\sstar{1^t}+P_3$ is a $\outcomeP$-position. We examine every option of the first player:
		\begin{itemize}
			\item If the first player plays to $\sstar{1^t} + P_2$ (resp. to $\sstar{1^{t-1}}+P_3$), then the second player answers by taking one leaf from $\sstar{1^t}$ (resp. by reducing $P_3$ to $P_2$), leaving $\sstar{1^{t-1}}+P_2$, which is a $\outcomeP$-position by induction hypothesis;
			\item If the first player plays to either $P_3$ or $\sstar{1^t}$, then the second player answers by playing by emptying the remaining graph (this is always possible since $t\geq4$ and $t\leq N-1$ implies $N\geq3$); 
			\item If the first player plays to either $\sstar{1^t} + P_1$ or $P_1 + P_3$, then the second player answers by playing to $P_1 + P_1$ which is a $\outcomeP$-position
		\end{itemize}
		
		\noindent\textbf{Case 3:} $t\geq N$
		
		We write $t=N+i$ with $i\geq 0$. We use induction on $i$ to prove that $\grundy(\sstar{1^t})=i \bmod 2$.
		
		If $i=0$, then $\sstar{1^t}$ has $N+1$ vertices. By Lemma~\ref{lem:sizeOfGraph}, this graph has Grundy value 0.
		
		If $i>0$, then there is only one move from $\sstar{1^t}$, which is removing one leaf, leaving $\sstar{1^{t-1}}$. Thus, we have:
		$$\begin{array}{lll}
		\grundy(\sstar{1^t}) & = \mex(\grundy(\sstar{1^{t-1}})) & \mbox{ by definition of } \grundy \\
		& = \mex(i-1 \bmod 2) & \mbox{ by induction hypothesis} \\
		& = i \bmod 2&
		\end{array}
		$$
	\end{proof}
	
	One can get the values for $k>0$ from these first values by recursive computation. For instance, table \ref{tab:grundyS1tkI4} gives the Grundy values for \csg{I_4} on $\sstar{1^t,k}$ for $t\leq 10$ and $k\leq 8$.
	
	
	\begin{table}[!h]
		\centering
		\begin{tabular}{c|c|c|c|c|c|c|c|c|c|c|c|}
			\diagbox{$k$}{$t$} & 0 & 1 & 2 & 3 & 4 & 5 & 6 & 7 & 8 & 9 & 10  \\ \hline
			0 & 1 & 2 &  3 &  2 &  0 &  1 &  0 &  1 &  0 &  1 &  0  \\ \hline
			1 & 2 & 3 & 2 & 0 & 1 & 0 & 1 & 0 & 1 & 0 & 1 \\ \hline
			2 &3 & 4 & 0 & 1 & 2 & 3 & 2 & 3 & 2 & 3 & 2  \\ \hline
			3 &4 & 0 & 1 & 4 & 3 & 2 & 3 & 2 & 3 & 2 & 3  \\ \hline
			4&0 & 1 & 5 & 3 & 4 & 5 & 4 & 5 & 4 & 5 & 4  \\ \hline 
			5&1 & 2 & 3 & 2 & 0 & 1 & 0 & 1 & 0 & 1 & 0  \\ \hline
			6&2 & 3 & 2 & 0 & 1 & 0 & 1 & 0 & 1 & 0 & 1  \\ \hline
			7&3 & 4 & 0 & 1 & 2 & 3 & 2 & 3 & 2 & 3 & 2  \\ \hline
			8&4 & 0 & 1 & 4 & 3 & 2 & 3 & 2 & 3 & 2 & 3  \\ \hline
		\end{tabular}
		\caption{The Grundy values of $\sstar{1^t,k}$ for the game \csg{I_4}, for $t \leq 10$ and $k \leq 8$.}
		\label{tab:grundyS1tkI4}
	\end{table}

	\subsection{Subdivided stars with three paths}
	
	We now focus on subdivided stars with three paths, with one of the path of size~$1$. We recall that we denote by $S_{1,k,\ell}$ the subdivided star with three paths of size respectively 1, $k$ and $\ell$. Technically speaking, all degree 1 vertices of $S_{1,k,\ell}$ are leaves, but, for the sake of simplicity, in the sequel, we will denote the degree 1 vertex corresponding to the path of size 1 as \emph{the} leaf. Conversely, we will use the word \emph{path} only to denote either $P_k$ or $P_{\ell}$, even if, technically speaking, $P_1$ is also a path. We prove that the Grundy values are purely periodic with period $N+1$ except for some values for which there is a preperiod $N+1$. More precisely, we will prove in this subsection the following theorem.

	\begin{theorem}
		\label{thm:S1lk}
		Let $N\geq 3$, $k,\ell$ be two nonnegative integers. The Grundy value of the subdivided star $\sstar{1,k,\ell}$ for the game \csg{I_n} is $$\grundy_{I_N}(\sstar{1,k,\ell})=\left\{
		\begin{array}{ll}
		\grundy_{I_N}(\sstar{1,(k \bmod (N+1)) +N+1,N}) & \mbox{if } k>N \mbox{ and } \ell \equiv N \bmod (N+1) \\
		\grundy_{I_N}(\sstar{1,N,(\ell \bmod (N+1)) +N+1}) & \mbox{if } \ell>N \mbox{ and } k \equiv N \bmod (N+1) \\
		\grundy_{I_N}(\sstar{1,k \bmod (N+1),\ell \bmod (N+1)}) & \mbox{otherwise}
		\end{array}
		\right.$$
	\end{theorem}
	
	Thanks to this theorem, for a fixed $N$ there are only $O(N^2)$ values to compute.
	
	\begin{cor}
		\label{cor:S1lk-polAlgo}
		Let $k,\ell$ be two nonnegative integers. Then there exists an algorithm computing the Grundy value of the subdivided star $\sstar{1,k,\ell}$ for the game \csg{I_N}, and this algorithm runs in time $\theta(N^2)$.
	\end{cor}

	However, all of these values are needed to prove Theorem \ref{thm:S1lk}. First, we give the Grundy values of $\sstar{1,k,\ell}$ when $\ell,k<N$.
	
	\begin{lemma}
		\label{lem:S1lkWithSmalllk}
		Let $k,\ell$ be nonnegative integers. 
		If $k,\ell<N$, then the Grundy value of the subdivided star $\sstar{1,k,\ell}$ for the game \csg{I_n} is  $$\grundy_{I_N}(\sstar{1,k,\ell})=\left\{
		\begin{array}{ll}
		k+\ell & \mbox{if } k+\ell \leq N-2 \mbox{ and } k,\ell \mbox{ odd} \\
		k+1 & \mbox{if } k+\ell \leq N-2 \mbox{ and } k=\ell\neq0 \mbox{ even} \\
		k+\ell+2 \bmod (N+1) & \mbox{otherwise}
		\end{array}
		\right.$$
	\end{lemma}
	
	These values are compiled in Table~\ref{tab:grundyValuesS1lk}.
	
	\begin{table}[!h]
		\centering
		\scalebox{0.8}{
			\begin{tabular}{|c|c|c|c|c|c|c|c|c|c|c|c|}
				\hline
				\diagbox{$k$}{$\ell$} & 0 & 1 & 2 & 3 & 4 & 5 & \ldots & $N-4$ & $N-3$ & $N-2$ & $N-1$ \\ \hline
				0 & 2 & 3 & 4 & 5 & 6 & 7 & \ldots & $N-2$ & $N-1$ & $N$ & 0 \\ \hline
				\multirow{2}{*}{1} & \multirow{2}{*}{3} & \multirow{2}{*}{\textbf{2}} & \multirow{2}{*}{5} & \multirow{2}{*}{\textbf{4}} & \multirow{2}{*}{7} & \multirow{2}{*}{\textbf{6}} & \multirow{2}{*}{\ldots} & \textbf{$\mathbf{N-3}$} if $N$ even & $N$ if $N$ even & \multirow{2}{*}{0} & \multirow{2}{*}{1}  \\ 
				& & & & & & & & $N-1$ if $N$ odd & $\mathbf{N-2}$ if $N$ odd & & \\ \hline
				2 & 4 & 5 & \textbf{3} & 7 & 8 & 9 & \ldots & $N$ & 0 & 1 & 2 \\ \hline
				3 & 5 & \textbf{4} & 7 & \textbf{6} & 9 & \textbf{8} & \ldots & 0 & 1 & 2 & 3 \\ \hline
				4 & 6 & 7 & 8 & 9 & \textbf{5} & 11 & \ldots & 1 & 2 & 3 & 4 \\ \hline
				\ldots & \ldots & \ldots & \ldots & \ldots & \ldots & \ldots & \ldots & \ldots & \ldots & \ldots & \ldots \\ \hline
				$N-2$ & $N$ & 0 & 1 & 2 & 3 & 4 & \ldots & $N-5$ & $N-4$ & $N-3$ & $N-2$ \\ \hline
				$N-1$ & 0 & 1 & 2 & 3 & 4 & 5 & \ldots & $N-4$ & $N-3$ & $N-2$ & $N-1$ \\ \hline
			\end{tabular}
		}
		\caption{Grundy values of $\sstar{1,k,\ell}$ for the game \csg{I_n}. The values where $\grundy_{I_N}(\sstar{1,k,\ell})\neq|\sstar{1,k,\ell}|\bmod(N+1)$ are bolded.}
		\label{tab:grundyValuesS1lk}
	\end{table}
	
	\begin{proof}
		In order to prove this result, we use induction on $(k,\ell)$ with lexicographic order.
		If $k$ or $\ell=0$, $\sstar{1,k,\ell}$ is simply the path $P_{\ell+k+2}$ on $\ell+k+2$ vertices, having Grundy value $k+\ell+2 \bmod (N+1)$.
		
		Assume now that $k,\ell \neq 0$, and assume that the formula holds for all $(i,j) < (k,\ell)$. We prove that the Grundy value of $\sstar{1,k,\ell}$ satisfies the formula of Lemma~\ref{lem:S1lkWithSmalllk}.
		
		By symmetry, we can assume that $k\leq \ell$.
		
		\medskip
		
		\noindent\textbf{Case 1:} $k+\ell \leq N-2$.
		
		From $\sstar{1,k,\ell}$, then there are exactly three types of possible moves:
		
		\begin{itemize}
			\item Type 1 move: removing the leaf, leaving $P_{k+\ell+1}$ which has Grundy value $k+\ell+1$ (only one such move).
			\item Type 2 move: removing the leaf, the central vertex, one of the two paths and some vertices from the other path, leaving a path $P_{i}$, for $i \in \llbracket 0,\ell \rrbracket$ (hence Type 2 moves can have all the Grundy values between $0$ and $\ell$).
			\item Type 3 move: removing some vertices from one of the paths, leaving $\sstar{1,i,\ell}$ for $i \in \llbracket 0,k-1 \rrbracket$, or  leaving $\sstar{1,k,j}$ for $j \in \llbracket 0,\ell-1 \rrbracket$, whose Grundy values may be computed by induction.
		\end{itemize}
		
		Depending on the parities of $k$ and $\ell$, the Grundy values of Type 3 moves may differ. The following subcases are dedicated to computing these Grundy values.
		
		\medskip

		\textbf{Subcase 1.1:} $k,\ell$ are odd.
		
		By induction hypothesis, the options $\sstar{1,i,\ell}$ with $i\in \llbracket 0,k-1\rrbracket$ have Grundy values:
		\begin{itemize}
			\item $\ell+i$ when $i$ is odd, giving all even values between $\ell+1$ and $\ell+k-2$;
			\item $\ell+i+2$ when $i$ is even, giving all odd values between $\ell+2$ and $\ell+k+1$.
		\end{itemize}
		
		That is to say, these options have all the values between $\ell+1$ and $\ell+k-1$, as well as value $\ell+k+1$.
		Similarly, the options $\sstar{1,k,j}$ with $j\in \llbracket 0,\ell-1 \rrbracket$ have all the Grundy values between $k+1$ and $\ell+k-1$, as well as value $\ell+k+1$.
		
		Finally, among the options of $\sstar{1,k,\ell}$, we have all the values between $0$ and $\ell+k-1$, except the value $k+\ell$, hence $\grundy(\sstar{1,k,\ell})=k+\ell$.
		
		\medskip

		\textbf{Subcase 1.2:} $k=\ell$ and $k$ is even.
		
		By induction hypothesis, the option $\sstar{1,i,k}$ have Grundy values $k+i+2>k+1$, for all $i\in \llbracket 0,k-1 \rrbracket$.
		Type 2 moves imply that $\grundy(\sstar{1,k,k}) \geq k+1$. Since no option of $\sstar{1,k,k}$ has Grundy value $k+1$, this means that $\grundy(\sstar{1,k,k})=k+1$.
		
		\medskip

		\textbf{Subcase 1.3:} $k+\ell \leq N-2$, $k$ odd, $\ell$ even.
		
		By induction hypothesis, the options $\sstar{1,i,\ell}$ with $i\in \llbracket 0,k-1\rrbracket$ have Grundy values $\ell+i+2$, giving all the values between $\ell+2$ and $\ell+k+1$ (since $i<k<\ell$ the case $\ell=i$ does not happen). As before, the options $\sstar{1,k,j}$ with $j\in \llbracket 0,j-1 \rrbracket$ give all the values between $k+1$ and $k+\ell$ and in particular the value $\ell+1$ since $k\geq 1$.
		Finally, all the values up to $\ell+k+1$ appear as options, except the value $\ell+k+2$, leading to $\grundy(\sstar{1,k,\ell})=k+\ell+2$.
		
		\medskip
		
		\textbf{Subcase 1.4:} $k+\ell \leq N-2$, $k$ even, $\ell$ odd.
		
		By induction, the options $\sstar{1,i,\ell}$  give all the values between $\ell+1$ and $\ell+k$. Hence all the values between $0$ and $\ell+k$ are present in the options. Remember that value $k+\ell+1$ is done by the Type 1 move. Finally, the Grundy value of options $\sstar{1,k,j}$ are smaller than $k+\ell+1$, and we have again $\grundy(\sstar{1,k,\ell})=k+\ell+2$.
		
		\medskip
		
		\textbf{Subcase 1.5:} $k+\ell \leq N-2$, $k$ even, $\ell$ even, $k\neq \ell$.
		
		The options $\sstar{1,i,\ell}$  give all the values between $\ell+2$ and $\ell+k+1$. The value $\ell+1$ is given by option $\sstar{1,k,\ell-k-1}$ (which exists since $k<\ell$). Thus all the values between $0$ and $\ell+k+1$ are present in the options. Finally, the Grundy value of options $\sstar{1,k,j}$ are smaller than $k+\ell+1$. Thus, we have again $\grundy(\sstar{1,k,\ell})=k+\ell+2$.

		\medskip
		
		\noindent\textbf{Case 2:} $k+\ell \geq N-1$ and $1 \leq k,\ell \leq N-1$.

		Let $j=k+\ell+2 - N+1$. Then $0\leq j \leq N-1$. We prove that $\sstar{1,k,\ell}+P_j$ is a $\outcomeP$-position.
		
		Assume that this is not true and take the smallest $j$ for which this result is not true.
		
		If the first player takes $c$ vertices in $\sstar{1,k,\ell}$ then the second player can always take $N+1-c$ vertices in the same component and move to $P_j$. Then the sum is a $\outcomeP$-position.
		Indeed, if the first player takes $c$ vertices on the path of size $k$ and $c\leq k$, then the second player can take the $k-c\leq N-2$ vertices remaining on the path of size $k$, the central vertex, the leaf, and some vertices on the path of size $\ell$ to obtain $P_j$. If the first player plays from $\sstar{1,k,\ell}$ to a path, this one has $k+\ell+2-c>j$ vertices and the second player can take $N+1-c$ vertices on it to go to $P_j$.
		
		Assume now that the first player plays $c$ vertices on $P_j$ to $P_{j'}$. By definition of $j$, we must have either $k\geq j$ or $\ell\geq j$. Assume that $k\geq j$. Then the first player takes $c$ vertices on the path of size $k$ to play to $\sstar{1,k',j}$. We still have $k'+\ell \geq N-1$ and by minimality, $\sstar{1,k',\ell}+P_{j'}$ is a losing position.

	\end{proof}

	The following lemma gives the periodicity of values 0 and 1. It completes Lemma \ref{lem:sizeOfGraph} for $S_{1,k,l}$.

	\begin{lemma}\label{lem:01s1kl}
		Let $k,\ell$  be two integers.
		
		If $|S_{1,k,\ell}|=0 \bmod (N+1)$ then $\grundy_{I_N}(S_{1,k,\ell})=0$.
		
		If $|S_{1,k,\ell}|=1 \bmod (N+1)$ then $\grundy_{I_N}(S_{1,k,\ell})=1$.
		
		Otherwise, $\grundy_{I_N}(S_{1,k,\ell})>1$.
	\end{lemma}
	
	\begin{proof}
		
		We prove the whole lemma by induction on $(k,l)$ (in lexicographic order). Note that this result is trivially true for paths. By Lemma \ref{lem:sizeOfGraph}, the result is true if $|S_{1,k,\ell}|=k+\ell+2\leq N+2$.
		
		Let $k,\ell$ with $k+\ell+2>N+2$.
		Assume first that $k+\ell+2=0\bmod (N+1)$. All the moves lead to a graph which is a path or a subdivided star with a number of vertices not congruent to 0 modulo $N+1$, thus, by induction the Grundy value is positive. Hence, $\grundy_{I_N}(S_{1,k,\ell})=0$.
		
		Assume now that $k+\ell+2=1\bmod (N+1)$. As in the previous case, all the moves lead to a graph which is a path or a subdivided star with a number of vertices not congruent to 1 modulo $N+1$, and by induction the Grundy value is not equal to 1. Furthermore, it is always possible to remove one vertex, leading to a graph with 0 vertices modulo $N+1$ that has by induction Grundy value 0. Hence, $\grundy_{I_N}(S_{1,k,\ell})=1$.
		
		Finally, assume that $k+\ell+2=i\bmod (N+1)$ with $i>1$.
		If one cannot remove $i$ vertices, it means that $k=\ell=i-1$. Hence the total number of vertices is $2i$ and must be congruent to $i$ modulo $N+1$. Since $i\neq 0$, this is not possible. Hence one can remove $i$ vertices and then the resulting graph has Grundy value 0.
		Similarly, if one cannot remove $i-1$ vertices, it means that $k=\ell=i-2$. Hence the total number of vertices is $2i-2$ and must be congruent to $i$ modulo $N+1$. It means that $i=2 \bmod (N+1)$ and so $i=2$. But then the graph is just $K_2$, which is not possible since there are at least $N+2$ vertices. Hence one can  always remove $i-1$ vertices and the resulting graph has Grundy value 1. Finally, $\grundy_{I_N}(S_{1,k,\ell})>1$.
	\end{proof}

	We now prove the "otherwise" part of Theorem \ref{thm:S1lk} which is the simpliest case when both paths can be reduced modulo $N+1$.

	\begin{lemma}
		\label{lem:S1lkInternalPeriodicity}
		Let $N\geq 3$, $k,\ell$ be two nonnegative integers satisfying the property:
		
		\begin{center}
			$(P)$ \quad $k\leq N$ or $\ell\leq N$ or $\ell$ and $k$ are not congruent to $N$ modulo $N+1.$
		\end{center}
		Then, we have $\grundy_{I_N}(\sstar{1,k,\ell})=
		\grundy_{I_N}(\sstar{1,k \bmod (N+1),\ell \bmod (N+1)})$.
	\end{lemma}
	
	\begin{proof}
		Let $k,\ell,N$ be integers satisfying $(P)$. Let $k_0=k \bmod{N+1}$ and $\ell_0=\ell \bmod{N+1}$.
		
		We want to prove that $\sstar{1,k,\ell}+\sstar{1,k_0,\ell_0}$ is a $\outcomeP$-position. Consider a minimal counterexample, when minimizing $k+\ell$. We consider all the possible winning moves for the first player.
		
		\medskip
		
		{\bf Case 1:} The first player plays in the first component $\sstar{1,k,\ell}$.
		
		\medskip
		
		\noindent {\bf Subcase 1.1:} The first player takes $c$ vertices in the path of size $\ell$ leading to $\sstar{1,k,\ell'}$ with $\ell'<\ell$.

		If $\ell>\ell_0$, then the second player can answer in the first component by $\sstar{1,k,\ell-(N+1)}$. Since $\ell-(N+1)=\ell_0 \bmod{N+1}$ and since $k,\ell-(N+1)$ satisfy $(P)$, by minimality the sum is a $\outcomeP$-position.
		
		Assume now that $\ell=\ell_0$. Then the second player answers in the second component to $\sstar{1,k_0,\ell'}$. Since $k,\ell'$ satisfy $(P)$ ($\ell'\leq N$), then the sum is by minimality a $\outcomeP$-position.
		
		By permuting $k$ and $\ell$, we can prove the symmetric case.
		\medskip
		
		\noindent {\bf Subcase 1.2:} The first player removes the leaf and thus move the first component to $P_{k+\ell+1}$. Then the second player can move the second component to $P_{k_0+\ell_0+1}$. Since $k=k_0 \bmod (N+1)$ and $\ell=\ell_0 \bmod (N+1)$, then these two paths have the same Grundy value and their sum is a $\outcomeP$-position.
		
		\medskip
		
		\noindent {\bf Subcase 1.3:} The first player removes the leaf, the path of size $k$, the central vertex and some vertices on the path of size $\ell$, leading the first component to $P_i$ with $i\leq \ell$. Then $k=k_0$ and the second player can remove the same number of vertices, leading to $P_{i-(\ell-\ell_0)}$, which has the same Grundy value as $P_i$.
		
		By permuting $k$ and $\ell$ we get the symmetric case.
		
		\medskip
		
		{\bf Case 2:} The first player plays in the second component $\sstar{1,k_0,\ell_0}$.
		
		\medskip
		
		\noindent {\bf Subcase 2.1:} The first player takes $c$ vertices in the path of size $\ell_0$ leading to $\sstar{1,k_0,\ell'_0}$ with $\ell'_0<\ell_0$.

		Let $\ell'=\ell-c$. Then $k,\ell'$ satisfy $(P)$. Indeed, $\ell'\neq N \bmod{N+1}$ since $\ell'=\ell'_0\bmod {N+1}$ and $0\leq \ell'_0<N$.
		Therefore, if the second player answer by playing in the first component to $\sstar{1,k,\ell'}$ we get a $\outcomeP$-position by minimality.
		
		By permuting $k$ and $\ell$, we can prove the symmetric case.
		\medskip
		
		\noindent {\bf Subcase 2.2:} The first player removes the leaf and thus moves the second component to $P_{k_0+\ell_0+1}$. Then the second player can move the first component to $P_{k+\ell+1}$. Since $k=k_0 \bmod (N+1)$ and $\ell=\ell_0 \bmod (N+1)$, then these two paths have the same Grundy value and their sum is a $\outcomeP$-position.
		
		\medskip
		
		\noindent {\bf Subcase 2.3:} The first player removes the leaf, the path of size $k_0$, the central vertex and some vertices on the path of size $\ell_0$, leading the first component to $P_i$ with $i\leq \ell_0$. Let $c$ be the total number of vertices taken by the first player.
		If $k=k_0$, then the second player can play in the first component to $P_{i+\ell-\ell_0}$, which has the same Grundy value as $P_i$.
		Hence we can assume that $k>k_0$. 
		Consider the position $\sstar{1,k-c,\ell}$.
		
		Let $k'_0=k-c \bmod (N+1)$. We  have $k'_0=k_0-c+N+1$ and $i+c=k_0+\ell_0+2$.
		Then $k'_0+\ell_0=k_0-c+N+1+\ell_0 =i+N+1$. In particular, $k'_0+\ell_0\geq N-1$.
		Hence, we are in the last case of Lemma \ref{lem:S1lkWithSmalllk}, and, by induction, $$\grundy(\sstar{1,k-c,\ell})=\grundy(\sstar{1,k'_0,\ell_0})=k'_0+\ell_0+2 \bmod (N+1)=i=\grundy(P_i).$$
		
		By permuting $k$ and $\ell$, we can prove the symmetric case.
	\end{proof}
	
	We will now study the case where $(P)$ is not satisfied, which are the "if" parts of Theorem~\ref{thm:S1lk}. Since both parts are symmetrical, we only need to prove one of them.
	
	\begin{lemma}
		\label{lem:S1lkExternalPeriodicity}
		Let $N\geq 3$, $k,\ell$ be two nonnegative integers such that $k>N$ and $\ell \equiv N \bmod (N+1)$. Then, we have $\grundy_{I_N}(\sstar{1,k,\ell})=\grundy_{I_N}(\sstar{1,(k \bmod (N+1))+N+1,N})$.
	\end{lemma}
	
	\begin{proof}
		
		We will first prove that the options $\sstar{1,i,N}$ with $i<N-2$ of $\sstar{1,k,N}$ (obtained by playing on the branch of size $N$) are not useful to compute $\grundy(\sstar{1,k,N})$, when $N<k<2N$ . In order to prove this, we show how to compute the Grundy value of $\sstar{1,k,N}$ for $k \in \llbracket 1,2N \rrbracket$. 
		
		Let $r_N = \left \lfloor \frac{N}{2} \right \rfloor$ if $N \equiv 2,3 \bmod 4$, and $\left \lfloor \frac{N-2}{2} \right \rfloor$ otherwise. 
		
		This first technical claim will let us make explicit the Grundy values of the options of $\sstar{1,k,N}$ reached when playing on the path of length $N$:
		
		
		
		
		\begin{claim}
			\label{clm:valuesOfColumn}
			Let $k \in \llbracket 1,2N-1 \rrbracket$ be a positive integer with $k \not\in \{N-1,N,N+1\}$, and $k_0 = k \bmod (N+1)$.
			\begin{enumerate}
				\item If $k_0 \in \llbracket 1,r_N \rrbracket$ is even, then $\sstar{1,k,N}$ has options of all Grundy values in the interval $\llbracket 0,N \rrbracket$, except $2k_0+2$, by removing vertices from the path of length $N$.
				\item If $k_0 \in \llbracket r_N+1,N-2 \rrbracket$ is even, then $\sstar{1,k,N}$ has options of all Grundy values in the interval $\llbracket 0,N \rrbracket$, except $k_0+1$, by removing vertices from the path of length $N$.
				\item If $k_0$ is odd, then $\sstar{1,k,N}$ has options of all Grundy values in the interval $\llbracket 0,N \rrbracket$, except $N-1$ (if $N$ is odd), or $N$ (if $N$ is even), by removing vertices from the path of length $N$.
			\end{enumerate}
		\end{claim}
		
		\begin{proof}
			Let $k \in \llbracket 1,N-2 \rrbracket$ be a positive integer.
			We will make use of the Grundy values given in Lemma~\ref{lem:S1lkWithSmalllk}.
			
			If $k \in \llbracket 1,r_N \rrbracket$ is even, then every subdivided star $\sstar{1,k,i}$ (with $i \in \llbracket 0,N-1 \rrbracket$) has Grundy value $i+k+2 \bmod (N+1)$, except the subdivided star $\sstar{1,k,k}$ which has Grundy value $k+1$. Note that we have $2k+2 \leq N-2$ (by definition of $r_N$ and since $k$ is even).
			Thus, the subdivided star $\sstar{1,k,N}$ has options with all the Grundy values in the set $\{k+2 \bmod (N+1), \ldots , k+1+N \bmod (N+1)\}$, except $2k+2$, and an option of Grundy value $k+1$. This proves the first part of the claim.
			
			Now, if $k \in \llbracket r_N+1,N-2 \rrbracket$ is even, then every subdivided star $\sstar{1,k,i}$ (with $i \in \llbracket 0,N-1 \rrbracket$) has Grundy value $i+k+2 \bmod (N+1)$.
			Thus, the subdivided star $\sstar{1,k,N}$ has options with all the Grundy values in the set $\{k+2 \bmod (N+1), \ldots , k+1+N \bmod (N+1)\}$, proving the second part of the claim.
			
			Finally, if $k$ is odd, then the subdivided stars $\sstar{1,k,i}$ have the following Grundy values: $i+k$ if $i < N-k-1$ and $i$ is odd; and $i+k+2 \bmod (N+1)$ otherwise.
			Thus, the subdivided star $\sstar{1,k,N}$ has options with all the Grundy values in the set $\{0,\ldots,k\}$ (when $i \geq N-k-1$), and in the set $\{k+1, \ldots, N-2\}$. Furthermore, if $N$ is even then it has an option of Grundy value $N-1$, and if $N$ is odd then it has an option of Grundy value $N$. Indeed, the Grundy values of $\sstar{1,k,0},\sstar{1,k,1},\sstar{1,k,2},\ldots,\sstar{1,k,N-k-4},\sstar{1,k,N-k-3},\sstar{1,k,N-k-2}$ are $k+2,k+1,k+4,\ldots$ and either $\ldots,N-4,N-1,N-2$ if $N$ is even or $\ldots,N-2,N-3,N$ if $N$ is odd. This proves the third part of the claim.
			
			If $k \in \llbracket N+2,2N-1 \rrbracket$, then the result holds by Lemma~\ref{lem:S1lkInternalPeriodicity}: the options of $\sstar{1,k,N}$ when removing vertices from the path of length $N$ have the same Grundy values as the equivalent options of $\sstar{1,k-(N+1),N}$.
		\end{proof}
		
		We are now ready to compute the values of $\grundy(\sstar{1,k,N})$ for $k\in \llbracket 1,2N \rrbracket$. 
		
		If $k=a(N+1)+b$ with $a \geq 0$ and $b \in \llbracket 1,N-1 \rrbracket$, then let $x_k$ be the number of subdivided stars $\sstar{1,i,N}$ with $i \in \llbracket a(N+1)+1,k-1 \rrbracket$ such that $\grundy(\sstar{1,i,N})>N$.
		
		\begin{claim}\label{claim:2N}
			Let $N\geq 3$ be a nonnegative integer, and $k \in \llbracket 1,2N \rrbracket$, $k \notin \{N,N+1\}$. Let $m = \mex(\{\grundy(\sstar{1,k,i}) | i<N\})$.
			We have the following:
			$$
			\grundy(\sstar{1,k,N}) = \left\{
			\begin{array}{ll}
			N & \mbox{if } k=N-1 \mbox{ or } k=2N \\
			N-1 & \mbox{if } k=N-2 \mbox{ or } k=2N-1 \\
			N+x_k+1 & \mbox{if } k \in \llbracket 1,r_N \rrbracket \mbox{ or } m \geq N-1 \\
			m & \mbox{otherwise} 
			\end{array}
			\right.
			$$
		\end{claim}
		
		\begin{proof}
			We prove the result by way of contradiction. Assume that $k \in \llbracket 1,2N \rrbracket$ is the smallest positive integer such that the result does not hold. Let $k_0 = k \bmod (N+1)$.
			From $\sstar{1,k,N}$, there are four kinds of moves:
			\begin{itemize}
				\item Type 1 move: removing the leaf, leaving $P_{N+1+k}$ which has Grundy value $k \bmod (N+1)$ (thus this option will never lead to a contradiction);
				\item Type 2 move: removing vertices from the path of length $N$, leaving stars $\sstar{1,k,i}$ for $i \in \llbracket 0,N-1 \rrbracket$;
				\item Type 3 move: removing vertices from the path of length $k$, leaving stars $\sstar{1,i,N}$ for $i \in \llbracket max(0,k-N),k-1 \rrbracket$;
				\item Type 4 move: if $k \leq N-2$, removing the path of length $k$, the leaf, the central vertex, and some vertices from the path of length $N$.
			\end{itemize}
			
			First, note that if $k \geq N+1$, then $x_k \leq x_{k-(N+1)}$. Indeed, by minimality of $k$, there are at most as many subdivided stars $\sstar{1,i,N}$ with Grundy values greater than $N$ when $i \in \llbracket N+2,k-1 \rrbracket$ than when $i \in \llbracket 1,k-(N+1)-1 \rrbracket$.
			
			Observe that if $k=N$ (resp. $k=N+1$) then we have $|\sstar{1,N,k}|=2N+2$ (resp. $|\sstar{1,N,k}|=2N+3$) and by Lemma~\ref{lem:01s1kl} we have $\grundy(\sstar{1,N,k})=0$ (resp. $\grundy(\sstar{1,N,k})=1$). So $k \notin \{N,N+1\}$.
			
			Assume that $k_0=N-1$, then by Lemmas~\ref{lem:S1lkWithSmalllk} and~\ref{lem:S1lkInternalPeriodicity}, Type~2 moves leave subdivided stars with all Grundy values in $\llbracket 0,N-1 \rrbracket$, so $\grundy(\sstar{1,k,N}) \geq N$. Furthermore, Type~3 moves will leave subdivided stars with Grundy values either smaller or greater than $N$ by minimality of $k$. Thus, $\grundy(\sstar{1,k,N})=N$, a contradiction. So $k \notin \{N-1,2N\}$.
			
			Now let $k \in \llbracket 1,2N-1 \rrbracket$, $k \notin \{N-1,N,N+1\}$. Assume first that $k_0$ is odd. Then, by Claim~\ref{clm:valuesOfColumn}, the Type 2 moves leave subdivided stars with all Grundy values in $\llbracket 0,N \rrbracket$, except $N-1$ or $N$, depending on the parity of $N$.
			
			Now if $k_0 < N-2$, then there are two cases:
			\begin{enumerate}
				\item If $k < N-2$ then $\sstar{1,k,N}$ has options of Grundy values $N-1$ and $N$ by Type~4 moves: leaving $P_{N-1}$ requires removing $k+2+1 \leq N$ vertices, and leaving $P_N$ requires removing $k+2 < N$ vertices.
				\item If $k \in \llbracket N+2,2N-2 \rrbracket$, then Type 3 moves allow to reach $\sstar{1,N-2,N}$ and $\sstar{1,N-1,N}$, which have Grundy values $N-1$ and $N$.
			\end{enumerate}
			Thus, the Grundy value of $\sstar{1,k,N}$ is greater than $N$. Furthermore, the Type~3 moves enable to reach $x_k$ options of Grundy values $N+1,\ldots,N+x_k$ by definition of $x_k$. And, if $k \in \llbracket N+2,2N-1 \rrbracket$, since $x_k \leq x_{k-(N+1)}$, the subdivided star $\sstar{1,i,N}$ such that $\grundy(\sstar{1,i,N})=N+x_k+1$ verifies $i \leq k-(N+1)$, and as such is not an option of $\sstar{1,k,N}$. This implies that $\grundy(\sstar{1,k,N})=N+x_k+1$, a contradiction. Thus if $k_0$ is odd then $k_0 = N-2$.
			
			If $k_0 = N-2$, then $N$ is odd and $\sstar{1,k,N}$ has no option of Grundy value $N-1$ from Type~2 moves. Furthermore, Type~3 moves cannot lead to an option of Grundy value $N-1$.
			Indeed, those options have either $P_{N-1}$ (if $k=N-2$) or $\sstar{1,N-2,N}$ (if $k=2N-1$) as an option, and those have Grundy value $N-1$ by minimality of $k$. Those cannot be options of $\sstar{1,k,N}$ since it would require removing $N+1$ vertices.
			Finally, if $k=N-2$ then Type 4~moves can only leave $P_N$. Thus, $\grundy(\sstar{1,k,N})=N-1$, a contradiction.
			
			This implies that $k_0$ is even. Assume first that $k_0 \in \llbracket 1,r_N \rrbracket$. Then by Claim~\ref{clm:valuesOfColumn}, Type~2 leave subdivided stars with all Grundy values in $\llbracket 0,N \rrbracket$, except $2k+2$, thus $\grundy(\sstar{1,k,N}) \geq 2k+2$. We recall that $2k+2 \leq N-2$.
			There are two cases:
			\begin{enumerate}
				\item If $k \in \llbracket 1,r_N \rrbracket$, then the subdivided star $\sstar{1,k,N}$ has $P_{2k+2}$ as an option, by removing the path of length $k$, the leaf, the central vertex and $N-(2k+2)$ vertices from the path of length $N$. This move removes $N-k$ vertices, so this option is always available. Thus, the star $\sstar{1,k,N}$ has options of all Grundy values in $\llbracket 0,N \rrbracket$. Since Type~3 moves enable to reach $x_k$ options of Grundy values $N+1,\ldots,N+x_k$ by definition of $x_k$, we have $\grundy(\sstar{1,k,N})=N+x_k+1$, a contradiction.
				\item If $k \in \llbracket N-2,N+1+r_N \rrbracket$, then by minimality of $k$, no option of $\sstar{1,k,N}$ has Grundy value $2k_0+2$: the only options with a Grundy value smaller than $N$ have Grundy values either of the form $2i+2$ and are smaller than $2k_0+2$, or of the form $i+1$ with $i$ even, and thus are odd and different from $2k_0+2$. All this implies that $\grundy(\sstar{1,k,N}) = 2k_0+2$, a contradiction.
			\end{enumerate}
			
			Thus, we have $k_0 \in \llbracket r_N+1,N-2 \rrbracket$. By Claim~\ref{clm:valuesOfColumn} the second kind of move leaves stars with all Grundy values in $\llbracket 0,N \rrbracket$, except $k+1$, thus $\grundy(\sstar{1,k,N}) \geq k+1$. The third kind of move can never leave a star with Grundy value $k+1$. Indeed, if $k \in \llbracket r_N+1,N-2 \rrbracket$ then by minimality of $k$, if an option by the third kind of move has a Grundy value smaller than $N$ then it is also smaller than $k+1$. Furthermore, if $k \in \llbracket N+2+r_N,2N-1 \rrbracket$ then any option by the third kind of move has $\sstar{1,k_0,N}$ as an option, which implies that all these options have a Grundy value different from $k_0+1$. Finally, if $k \in \llbracket r_N+1,N-2 \rrbracket$ then the fourth kind of move cannot leave $P_{k+1}$ since it would require removing $k+2+N-(k+1)=N+1$ vertices, which is impossible. Thus, the star $\sstar{1,k,N}$ has no option of Grundy value $k+1$, and as such we have $\grundy(\sstar{1,k,N})=k+1$, a contradiction.
			
			All of this implies that no such $k$ exists, so the claim is proven.
		\end{proof}

		We now prove the lemma. Let $k,\ell$ be two nonnegative integers such that $k>N$ and $\ell \equiv N \bmod (N+1)$. Let $k_0=k\bmod (N+1) + N+1$.
		
		A consequence of the previous proof is that the options $\sstar{1,i,N}$ with $i<N-2$ of $\sstar{1,k_0,N}$ (obtained by playing on the path of size $k_0$) are not useful to compute $\grundy(\sstar{1,k_0,N})$.
		
		We prove the result by induction on $k+\ell$. If $\ell=N$ and $k=N+1$ the result is clearly true.
		
		Compare the options of $\sstar{1,k,\ell}$ and $\sstar{1,k_0,N}$.
		By Lemma \ref{lem:S1lkInternalPeriodicity}, options $\sstar{1,k,\ell-i}$ and $\sstar{1,k_0, N-i}$ have the same Grundy values.
		By induction and thanks to Claim \ref{claim:2N}, options $\sstar{1,k-i,\ell}$ and $\sstar{1,k_0-i,N}$, if $i\leq k_0+2$, have the same Grundy values.
		Options $P_{k+\ell+1}$ and $P_{k_0+N+1}$ have the same Grundy values by periodicity on the path (Lemma~\ref{lem:In-paths}).
		
		Thus the only different options are $\sstar{1,k-i,\ell}$ and $\sstar{1,k_0-i,N}$ for $k_0+2<i\leq N$, if $k>k_0$ (if $k=k_0$, then they are the same by Lemma \ref{lem:S1lkInternalPeriodicity} and thus we are done). By the previous remark, the second ones are not used to compute $\grundy(\sstar{1,k_0,N})$. The first ones are different from $\grundy(\sstar{1,k-(N+1),\ell})$ since one can play from $\sstar{1,k-i,\ell}$ to $\sstar{1,k-(N+1),\ell}$. By induction, $\grundy(\sstar{1,k-(N+1),\ell})=\grundy(\sstar{1,k_0,N})$ and thus, we also have $\grundy(\sstar{1,k,\ell})=\grundy(\sstar{1,k_0,N})$.

	\end{proof}

	The proof of Theorem~\ref{thm:S1lk} is given by Lemma~\ref{lem:S1lkInternalPeriodicity} and Lemma~\ref{lem:S1lkExternalPeriodicity}.
	Note that the Grundy sequence is not purely periodic in general since for specific values of $N$ and $i$, we have  $\grundy_{I_N}(\sstar{1,N,N+1+i})\neq \grundy_{I_N}(\sstar{1,N,i})$ (for example, $\grundy_{I_8}(\sstar{1,8,2})=10$ whereas $\grundy_{I_8}(\sstar{1,8,11})=6$). However, when $N$ is small enough, the "bad cases" do not happen and the sequence is purely periodic. This is the case for $N=3$. We will see in the next section, that actually the game \csg{I_3} is purely periodic on all the subdivided stars.

	\subsection{Application: study of \csg{1,2,3} on subdivided stars}
	
	In this subsection, we use the previous general results to completely solve the game \csg{1,2,3} on subdivided stars.

	\begin{theorem}\label{thm:123}
		For all $\ell_1,\ldots,\ell_t \geq0$, $t\geq0$, we have:
		$$\grundy_{I_3}(\sstar{\ell_1,\ldots,\ell_t}) = \grundy_{I_3}(\sstar{\ell_1 \bmod 4,\ldots,\ell_t \bmod 4}).$$ 
	\end{theorem}
	
	\begin{proof}
		We just need to prove that we can remove four vertices on any branch.
		Let $S=\sstar{\ell_1,\ldots,\ell_t,m}$ and $S'=\sstar{\ell_1,\ldots,\ell_t,m+4}$. We prove that $S+S'$ is a $\outcomeP$-position. If $S \in \{ \emptyset, P_1, P_2 \}$, then $S' \in \{ P_4, P_5, P_6 \}$ and the result holds by Lemma~\ref{lem:In-paths}. Hence we will suppose that $|S| \geq 3$.
		
		We prove the result by way of contradiction. Let $u$ be the leaf of the path of length $m$ in $S$ (if $m=0$, then $u$ is the central vertex of $S$). Assume $S$ is minimal in the sense that $S \not\equiv S'$ but $S_1 \equiv \app{S_1}{u}{4}$ for every sub-subdivided star $S_1$ of $S$. Note that, by sub-subdivided star, we mean a subdivided star $S_1 = \sstar{\ell_1',\ldots,\ell_t'}$ with $\ell_i' \leq \ell_i$ for all $i$, and with ant least one $i$ such that $\ell_i' < \ell_i$. By Proposition~\ref{prop:winning_move} the only winning move in $S+S'$ is in $S$, removes the vertex $u$, and leaves at least two vertices in $S$.
		
		If $u$ is the central vertex of $S$, then there are two cases:
		\begin{enumerate}
			\item $S$ is a path of length $k \geq 3$. In this case, $u$ can be a leaf, or at distance 1 or 2 from a leaf:
			\begin{enumerate}
				\item If $u$ is a leaf, then $S'$ is a path and $S \equiv S'$ by Lemma~\ref{lem:In-paths}, a contradiction.
				\item If $u$ is at distance 1 from a leaf, then $S'=\sstar{1,k-2,4}$. If $k-2 \not\equiv 3 \bmod 4$,  then by Theorem~\ref{thm:S1lk}, we have $\grundy(\sstar{1,k-2,4}) = \grundy(P_{k})$ and thus $S' \equiv S$.
				Otherwise, we have $k \equiv 1 \bmod 4$ and, as such, by Lemma~\ref{lem:01s1kl} we have $\grundy(S')=1=\grundy(P_k)$. Both cases are contradictions.
				\item If $u$ is at distance 2 from a leaf, then $S'=\sstar{2,k-3,4}$ and the winning move consists in removing the path of length 2 along with $u$ in $S$, leaving $P_{k-3}$. We claim that removing three vertices from the path of length 4 in $S'$ is an equivalent move. Indeed, the move leaves $\sstar{1,2,k-3}$. If $k-3 \not\equiv 3 \bmod 4$, then by Theorem~\ref{thm:S1lk} and Lemma~\ref{lem:S1lkWithSmalllk} we have $\grundy(\sstar{1,2,k-3}) =\grundy( \sstar{1,2,k-3 \bmod 4})= k+1 \bmod 4$ (the last equality comes from the fact that $2+(k-3 \bmod 4) \geq 2$). Finally, $\grundy(\sstar{1,2,k-3})=\grundy(P_{k-3})$. Otherwise, by Theorem~\ref{thm:S1lk}, we have $\grundy(\sstar{1,2,k-3})=\grundy(\sstar{1,2,3})$.
				By computation we have $\grundy(\sstar{1,2,3})=3=\grundy(P_{k-3})$. Both cases are contradictions.
			\end{enumerate}
			\item $S=\sstar{1,1,k}$. In this case, the winning move can only be removing the two leaves and $u$, leaving $P_k$. We claim that removing three vertices from the path of length 4 in $S'$ is an equivalent move. Indeed, this leaves $\sstar{1^3,k}$, which is equivalent to $\sstar{1^3,k \bmod 4}$ by Proposition~\ref{prop:Stk_general}. There are four possibles values for $k \bmod 4$:
			\begin{enumerate}
				\item If $k \bmod 4 = 0$ then we have $\grundy(\sstar{1^3,k})=\grundy(\sstar{1^3})=0$ by Lemma~\ref{lem:csg1nstarsareP}.
				\item If $k \bmod 4 = 1$ then we have $\grundy(\sstar{1^3,k})=\grundy(\sstar{1^4})=1$ by Lemma~\ref{lem:csg1nstarsareP}.
				\item If $k \bmod 4 = 2$ then there are three options for $\sstar{1^3,2}$ which are $\sstar{1^3}$ (which has Grundy value~0 by Lemma~\ref{lem:csg1nstarsareP}), $\sstar{1^4}$ (which has Grundy value~1 by Lemma~\ref{lem:csg1nstarsareP}) and $\sstar{1^2,2}$ (which has Grundy value~1 by Lemma~\ref{lem:S1lkWithSmalllk}). Thus $\grundy(\sstar{1^3,2})=2$.
				\item If $k \bmod 4 = 3$ then there are four options for $\sstar{1^3,3}$ which are $\sstar{1^3}$ (which has Grundy value 0 by Lemma~\ref{lem:csg1nstarsareP}), $\sstar{1^4}$ (which has Grundy value 1 by Lemma~\ref{lem:csg1nstarsareP}), $\sstar{1^3,2}$ (which has Grundy value 2 by the above subcase) and $\sstar{1^2,3}$ (which has Grundy value 4 by computation, and using Lemma \ref{lem:S1lkWithSmalllk}). Thus $\grundy(\sstar{1^3,2})=3$.
			\end{enumerate}
			
			In the four cases, we have $\grundy(\sstar{1^3,k \bmod 4}) = k \bmod 4 = \grundy(P_k)$. Thus the only possible winning move in $S+S'$ actually leads to a an $\outcomeN$-position. This is a contradiction.
		\end{enumerate}
		
		Assume now that $u$ is not the central vertex, which means that $u$ is a leaf. If the winning move of the first player does not take the central vertex, then the second player can take the same number of vertices in $S'$ in the branch of size $m+4$.
		So, necessarily, the winning move takes the central vertex, $u$, and leaves a path. This means that $S=\sstar{1,k,1}$ and that the first player plays to $P_k$.
		We claim that removing $3$ vertices from the path of length $4$ is equivalent to $P_k$. Indeed, the second component becomes $S'_1=\sstar{1,k,2}$ and by Theorem \ref{thm:S1lk} and Lemma~\ref{lem:S1lkWithSmalllk}, it has Grundy value $k\bmod {4}$.

		All the cases lead to a contradiction, which proves that there is no $S$ such that $S+S'$ is an $\outcomeN$-position. Hence, we have $\grundy(S)=\grundy(S')$.
	\end{proof}

	\section{The $CSG(1,2,4)$ game on subdivided stars}\label{sec:124}
	
	First, note that, if $M \not\equiv 0 \bmod (N+1)$, then the game \csg{I_N \cup \{ M \}} is equivalent to the game \csg{I_N} when those games are played on paths. However, this is not the case when played on graphs, and even on subdivided stars.
	
	\begin{obs}\label{obs:plus_M}
		For all $N \geq 2$, there exists $M \not\equiv 0 \bmod (N+1)$, such that there exists a graph $G$, and a distinguished vertex $u$ in $G$, such that $\grundy_L(G) \neq \grundy_L(\app{G}{u}{N+1})$ where 
		$L=I_N \cup \{ M \}$.
	\end{obs}
	
	\begin{proof}
		Let $S$ be a star with $N+2$ leaves and $u$ as a central vertex, and $M=2N+4$. $S + \app{S}{u}{N+1}$ is an $\outcomeN$-position.
		The strategy for the first player is to empty $\app{S}{u}{N+1}$ as a first move. The second player is left with a forced move on $S$. After two more forced moves, the first player will be able to empty $S$.
	\end{proof}
	
	However, some particular games may retain this equivalence.
	In this section we study the game \csg{1,2,4} on subdivided stars and prove the following, which is analogous to Theorem~\ref{thm:123}:
	
	\begin{theorem}\label{thm:124}
		Let $L=\{1,2,4\}$. For all $m\geq 0$ and $\ell_1,\ldots,\ell_t \geq0$, $t\geq0$, we have:
		$$\grundy_L(\sstar{\ell_1,\ldots,\ell_t,m+3}) = \grundy_L(\sstar{\ell_1,\ldots,\ell_t,m}).$$
	\end{theorem}
	
	In order to prove this theorem, we first derive results for some specific subclasses of subdivided stars, and we then proceed by induction. For the rest of this section, all Grundy values will be computed for the game \csg{1,2,4}.
	
	First, it is well-known and easy to see that this game is 3-periodic when played on paths.
	
	\begin{lemma}\label{lem:124-paths}
		We have $\grundy(P_n) = |P_n| \bmod 3$, i.e. $\grundy(P_n) = \overline{012}$, where $P_n$ denotes the path on $n$ vertices.
	\end{lemma}
	
	Let us begin by some specific small subdivided stars that will be useful in the sequel.
	
	\begin{lemma}\label{lem:124-small-ones}
		We have $\grundy(G)=|G| \bmod 3$ for all $G\in S$, with
		\begin{eqnarray*}
			S&=&\{\sstar{1,1,1}, \sstar{1,1,2}, \sstar{1,1,3}, \sstar{1,1,1,3},\\
			&& \sstar{1,2,2}, \sstar{1,2,3}, \sstar{1,1,2,2}, \sstar{1,1,2,3}\}.
		\end{eqnarray*}
		In addition, we have $\grundy(\sstar{1,1,1,1})=0$ and $\grundy(\sstar{1,1,1,2})=3$.
	\end{lemma}

	\begin{proof}
		The proof is straightforward and a consequence of simple case analysis. It is easy to check that $\grundy(\sstar{1,1,1})$ is indeed 1, since there are only 2 possible moves, that lead to $\emptyset$ and to the path $P_3$, that both have Grundy number 0. As for $\sstar{1,1,2}$, there are 4 possible moves, leading to paths $P_1$ and $P_3$, and to $\sstar{1,1,1}$, and the result follows from the application of the mex rule. Similarly, 4 moves are available from $\sstar{1,1,3}$, leading to paths and to $\sstar{1,1,1}$ and $\sstar{1,1,2}$. In the case $\sstar{1,1,1,1}$, we have only 2 possible moves, leading to $\sstar{1,1,1}$ and $P_1$, hence $\grundy(\sstar{1,1,1,1})=0$. From $\sstar{1,1,1,2}$, 4 moves are available, leading to a path $P_2$ and to $\sstar{1,1,1,1}$, $\sstar{1,1,1}$, and $\sstar{1,1,2}$. We use this result to compute $\grundy(\sstar{1,1,1,3})$, since one of the 4 possible moves leads to $\sstar{1,1,1,2}$, the other leading also to a path and to already known subdivided stars, namely $\sstar{1,1,1,1}$ and $\sstar{1,1,3}$. Similarly, there are 4 possible moves from $\sstar{1,2,2}$, all leading to already known graphs. There are 6 moves from $\sstar{1,2,3}$, again all leading to paths or previously studied subdivided stars. The cases $\sstar{1,1,2,2}$ and $\sstar{1,1,2,3}$ are similar: there are, respectively, 3 and 5 moves to study, that thankfully all lead to graphs whose Grundy value have been already computed above.
	\end{proof}
	
	\begin{lemma}\label{lem:124-S11k}
		For all $k\geq0$, we have $\grundy(\sstar{1,1,k})=|\sstar{1,1,k}| \bmod 3$, i.e. $\grundy(\sstar{1,1,k}) = \overline{012}$.
	\end{lemma}
	
	\begin{proof}
		We proceed by induction on $k$. The cases $k\leq3$ are all consequences of Lemma~\ref{lem:124-paths} and~\ref{lem:124-small-ones}. When $k\geq4$, we can apply Lemma~\ref{lem:general} since there are always exactly 5 moves available, leading to $P_{k+2}$, $P_{k-1}$, $\sstar{1,1,k-1}$, $\sstar{1,1,k-2}$, and $\sstar{1,1,k-4}$. Indeed, the lemma can be applied with $\mathcal{F}$ set to the set of all paths together with subdivided stars $\sstar{1,1,k'}$ such as $k'<k$, so as to get the desired result by induction (with $T$ in the statement of Lemma~\ref{lem:general} set to $3$ and $\alpha_i=i$ for $i\in\{0,1,2\}$).
	\end{proof}
	
	\begin{lemma}\label{lem:124-S111k}
		For all $k\geq0$, we have $\grundy(\sstar{1,1,1,k})=
		\left\{\begin{array}{c}
		1 \mbox{ if } k \equiv 0 \bmod 3\\
		0 \mbox{ if } k \equiv 1 \bmod 3\\
		3 \mbox{ if } k \equiv 2 \bmod 3
		\end{array}\right.$,
		i.e. $\grundy(\sstar{1,1,1,k}) = \overline{103}$.
	\end{lemma}
	
	\begin{proof}
		The small cases $k\leq3$ are all consequences of Lemma~\ref{lem:124-small-ones}. When $k\geq4$, there are always exactly 5 possible moves, leading to $\sstar{1,1,k}$, $P_{k}$, $\sstar{1,1,1,k-1}$, $\sstar{1,1,1,k-2}$, and $\sstar{1,1,1,k-4}$. By induction, and thanks to Lemma~\ref{lem:124-paths} and~\ref{lem:124-S11k}, we then have:
		\begin{itemize}
			\item If $k \equiv 0 \bmod 3$, then $\grundy(\sstar{1,1,1,k})$ is computed as $\mex(0,0,3,0,3)=1$, and we are done.
			\item If $k \equiv 1 \bmod 3$, then $\grundy(\sstar{1,1,1,k})$ is computed as $\mex(1,1,1,3,1)=0$, and we are done.
			\item If $k \equiv 2 \bmod 3$, then $\grundy(\sstar{1,1,1,k})$ is computed as $\mex(2,2,0,1,0)=3$, and we are done.
		\end{itemize}
	\end{proof}
	
	\begin{lemma}\label{lem:124-S12k}
		For all $k\geq0$, we have $\grundy(\sstar{1,2,k})=|\sstar{1,2,k}| \bmod 3$, i.e. $\grundy(\sstar{1,2,k}) = \overline{120}$.
	\end{lemma}
	
	\begin{proof}
		We proceed by induction on $k$. The cases $k\leq3$ are all consequences of Lemma~\ref{lem:124-paths} and~\ref{lem:124-small-ones}. When $k\geq4$, we can apply Lemma~\ref{lem:general} since there are always exactly 7 moves available, leading to $P_{k}$, $P_{k+2}$, $P_{k+3}$, $\sstar{1,1,k}$, $\sstar{1,2,k-1}$, $\sstar{1,2,k-2}$, and $\sstar{1,2,k-4}$. Indeed, the lemma can be applied with $\mathcal{F}$ set to the set of all paths together together with subdivided stars of Lemma~\ref{lem:124-S11k} and subdivided stars $\sstar{1,2,k'}$ such as $k'<k$, so as to get the desired result by induction (with $T$ in the statement of Lemma~\ref{lem:general} set to $3$ and $\alpha_i=i$ for $i\in\{0,1,2\}$).
	\end{proof}
	
	\begin{lemma}\label{lem:124-S112k}
		For all $k\geq0$, we have $\grundy(\sstar{1,1,2,k})=
		\left\{\begin{array}{c}
		2 \mbox{ if } k \equiv 0 \bmod 3\\
		3 \mbox{ if } k \equiv 1 \bmod 3\\
		1 \mbox{ if } k \equiv 2 \bmod 3
		\end{array}\right.$,
		i.e. $\grundy(\sstar{1,1,2,k}) = \overline{231}$.
	\end{lemma}
	
	\begin{proof}
		The small cases $k\leq3$ are all consequences of Lemma~\ref{lem:124-small-ones}. When $k\geq4$, there are always exactly 6 possible moves, leading to $\sstar{1,1,2,k-1}$, $\sstar{1,1,2,k-2}$, $\sstar{1,1,2,k-4}$, $\sstar{1,1,1,k}$, $\sstar{1,1,k}$, and $\sstar{1,2,k}$. By induction, and thanks to Lemma~\ref{lem:124-S11k},~\ref{lem:124-S111k} and~\ref{lem:124-S12k}, we then have:
		\begin{itemize}
			\item If $k \equiv 0 \bmod 3$, then $\grundy(\sstar{1,1,2,k})$ is computed as $\mex(1,3,1,1,0,1)=2$, and we are done.
			\item If $k \equiv 1 \bmod 3$, then $\grundy(\sstar{1,1,2,k})$ is computed as $\mex(2,1,2,0,1,2)=3$, and we are done.
			\item If $k \equiv 2 \bmod 3$, then $\grundy(\sstar{1,1,2,k})$ is computed as $\mex(3,2,3,3,2,0)=1$, and we are done.
		\end{itemize}
	\end{proof}
	
	Note that none of the subdivided stars $\sstar{1,1,2,k}$ is a $\outcomeP$-position.
	
	We also need the following detailed analysis of the subdivided star $\sstar{3,3,3}$.
	
	\begin{lemma}\label{lem:124-S333}
		We have $\grundy(\sstar{3,3,3})=\grundy(\sstar{2,2,2})=1$, $\grundy(\sstar{2,2,3})=2$, with the Grundy numbers of subdivided stars that are accessible from $\sstar{3,3,3}$ as in Figure~\ref{fig:124-S333}.
	\end{lemma}
	
	\begin{proof}
		The proof is straightforward, and relies on the previous Lemmas~\ref{lem:124-paths} and~\ref{lem:124-small-ones}. Figure~\ref{fig:124-S333} speaks for itself.
	\end{proof}
	
	
	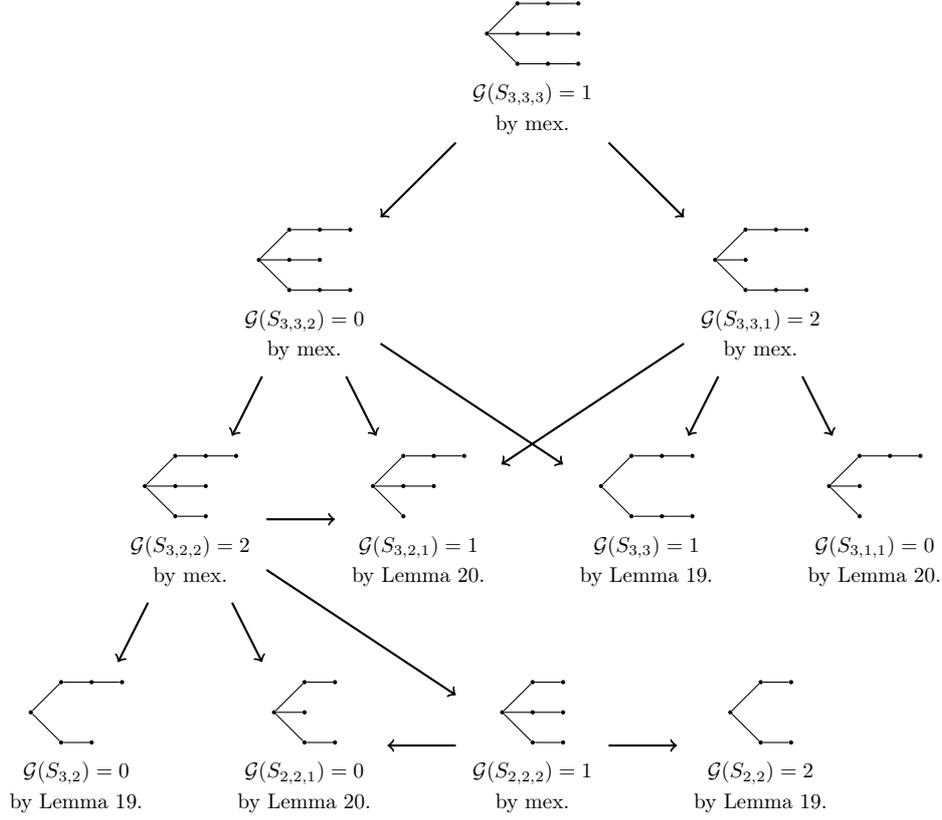
\begin{figure}[!h]
		\centering
		\begin{tikzpicture}
		\node (s333) at (0,0) {
			\begin{tikzpicture}[scale=0.8]
			\node (dessin) at (0,1) {
				\begin{tikzpicture}[scale=0.4]
				\node[noeud,scale=0.4] (c) at (0,0) {};
				\node[noeud,scale=0.4] (1) at (1,1) {};
				\node[noeud,scale=0.4] (2) at (2,1) {};
				\node[noeud,scale=0.4] (3) at (3,1) {};
				\node[noeud,scale=0.4] (4) at (1,0) {};
				\node[noeud,scale=0.4] (5) at (2,0) {};
				\node[noeud,scale=0.4] (6) at (3,0) {};
				\node[noeud,scale=0.4] (7) at (1,-1) {};
				\node[noeud,scale=0.4] (8) at (2,-1) {};
				\node[noeud,scale=0.4] (9) at (3,-1) {};
				\draw (c)--(1)--(3);
				\draw (c)--(7)--(9);
				\draw (c)--(6);
				\end{tikzpicture}
			};
			\node[scale=0.8] (a) at (0,0) {$\grundy(\sstar{3,3,3})=1$};
			\node[scale=0.8] (b) at (0,-0.5) {by $\mex$.};
		\end{tikzpicture}
	};
	\node (s332) at (-3,-3) {
		\begin{tikzpicture}[scale=0.8]
		\node (dessin) at (0,1) {
			\begin{tikzpicture}[scale=0.4]
			\node[noeud,scale=0.4] (c) at (0,0) {};
			\node[noeud,scale=0.4] (1) at (1,1) {};
			\node[noeud,scale=0.4] (2) at (2,1) {};
			\node[noeud,scale=0.4] (3) at (3,1) {};
			\node[noeud,scale=0.4] (4) at (1,0) {};
			\node[noeud,scale=0.4] (5) at (2,0) {};
			\node[noeud,scale=0.4] (7) at (1,-1) {};
			\node[noeud,scale=0.4] (8) at (2,-1) {};
			\node[noeud,scale=0.4] (9) at (3,-1) {};
			\draw (c)--(1)--(3);
			\draw (c)--(7)--(9);
			\draw (c)--(5);
			\end{tikzpicture}
		};
		\node[scale=0.8] (a) at (0,0) {$\grundy(\sstar{3,3,2})=0$};
		\node[scale=0.8] (b) at (0,-0.5) {by $\mex$.};
	\end{tikzpicture}
};
\node (s331) at (3,-3) {
	\begin{tikzpicture}[scale=0.8]
	\node (dessin) at (0,1) {
		\begin{tikzpicture}[scale=0.4]
		\node[noeud,scale=0.4] (c) at (0,0) {};
		\node[noeud,scale=0.4] (1) at (1,1) {};
		\node[noeud,scale=0.4] (2) at (2,1) {};
		\node[noeud,scale=0.4] (3) at (3,1) {};
		\node[noeud,scale=0.4] (4) at (1,0) {};
		\node[noeud,scale=0.4] (7) at (1,-1) {};
		\node[noeud,scale=0.4] (8) at (2,-1) {};
		\node[noeud,scale=0.4] (9) at (3,-1) {};
		\draw (c)--(1)--(3);
		\draw (c)--(7)--(9);
		\draw (c)--(4);
		\end{tikzpicture}
	};
	\node[scale=0.8] (a) at (0,0) {$\grundy(\sstar{3,3,1})=2$};
	\node[scale=0.8] (b) at (0,-0.5) {by $\mex$.};
\end{tikzpicture}
};
\node (s322) at (-4.5,-6) {
\begin{tikzpicture}[scale=0.8]
\node (dessin) at (0,1) {
	\begin{tikzpicture}[scale=0.4]
	\node[noeud,scale=0.4] (c) at (0,0) {};
	\node[noeud,scale=0.4] (1) at (1,1) {};
	\node[noeud,scale=0.4] (2) at (2,1) {};
	\node[noeud,scale=0.4] (3) at (3,1) {};
	\node[noeud,scale=0.4] (4) at (1,0) {};
	\node[noeud,scale=0.4] (5) at (2,0) {};
	\node[noeud,scale=0.4] (7) at (1,-1) {};
	\node[noeud,scale=0.4] (8) at (2,-1) {};
	\draw (c)--(1)--(3);
	\draw (c)--(7)--(8);
	\draw (c)--(5);
	\end{tikzpicture}
};
\node[scale=0.8] (a) at (0,0) {$\grundy(\sstar{3,2,2})=2$};
\node[scale=0.8] (b) at (0,-0.5) {by $\mex$.};
\end{tikzpicture}
};
\node (s321) at (-1.5,-6) {
\begin{tikzpicture}[scale=0.8]
\node (dessin) at (0,1) {
\begin{tikzpicture}[scale=0.4]
\node[noeud,scale=0.4] (c) at (0,0) {};
\node[noeud,scale=0.4] (1) at (1,1) {};
\node[noeud,scale=0.4] (2) at (2,1) {};
\node[noeud,scale=0.4] (3) at (3,1) {};
\node[noeud,scale=0.4] (4) at (1,0) {};
\node[noeud,scale=0.4] (5) at (2,0) {};
\node[noeud,scale=0.4] (7) at (1,-1) {};
\draw (c)--(1)--(3);
\draw (c)--(7);
\draw (c)--(5);
\end{tikzpicture}
};
\node[scale=0.8] (a) at (0,0) {$\grundy(\sstar{3,2,1})=1$};
\node[scale=0.8] (b) at (0,-0.5) {by Lemma~\ref{lem:124-small-ones}.};
\end{tikzpicture}
};
\node (s33) at (1.5,-6) {
\begin{tikzpicture}[scale=0.8]
\node (dessin) at (0,1) {
\begin{tikzpicture}[scale=0.4]
\node[noeud,scale=0.4] (c) at (0,0) {};
\node[noeud,scale=0.4] (1) at (1,1) {};
\node[noeud,scale=0.4] (2) at (2,1) {};
\node[noeud,scale=0.4] (3) at (3,1) {};
\node[noeud,scale=0.4] (7) at (1,-1) {};
\node[noeud,scale=0.4] (8) at (2,-1) {};
\node[noeud,scale=0.4] (9) at (3,-1) {};
\draw (c)--(1)--(3);
\draw (c)--(7)--(9);
\end{tikzpicture}
};
\node[scale=0.8] (a) at (0,0) {$\grundy(\sstar{3,3})=1$};
\node[scale=0.8] (b) at (0,-0.5) {by Lemma~\ref{lem:124-paths}.};
\end{tikzpicture}
};
\node (s311) at (4.5,-6) {
\begin{tikzpicture}[scale=0.8]
\node (dessin) at (0,1) {
\begin{tikzpicture}[scale=0.4]
\node[noeud,scale=0.4] (c) at (0,0) {};
\node[noeud,scale=0.4] (1) at (1,1) {};
\node[noeud,scale=0.4] (2) at (2,1) {};
\node[noeud,scale=0.4] (3) at (3,1) {};
\node[noeud,scale=0.4] (4) at (1,0) {};
\node[noeud,scale=0.4] (7) at (1,-1) {};
\draw (c)--(1)--(3);
\draw (c)--(7);
\draw (c)--(4);
\end{tikzpicture}
};
\node[scale=0.8] (a) at (0,0) {$\grundy(\sstar{3,1,1})=0$};
\node[scale=0.8] (b) at (0,-0.5) {by Lemma~\ref{lem:124-small-ones}.};
\end{tikzpicture}
};
\node (s222) at (0,-9) {
\begin{tikzpicture}[scale=0.8]
\node (dessin) at (0,1) {
\begin{tikzpicture}[scale=0.4]
\node[noeud,scale=0.4] (c) at (0,0) {};
\node[noeud,scale=0.4] (1) at (1,1) {};
\node[noeud,scale=0.4] (2) at (2,1) {};
\node[noeud,scale=0.4] (4) at (1,0) {};
\node[noeud,scale=0.4] (5) at (2,0) {};
\node[noeud,scale=0.4] (7) at (1,-1) {};
\node[noeud,scale=0.4] (8) at (2,-1) {};
\draw (c)--(1)--(2);
\draw (c)--(7)--(8);
\draw (c)--(5);
\end{tikzpicture}
};
\node[scale=0.8] (a) at (0,0) {$\grundy(\sstar{2,2,2})=1$};
\node[scale=0.8] (b) at (0,-0.5) {by $\mex$.};
\end{tikzpicture}
};
\node (s221) at (-3,-9) {
\begin{tikzpicture}[scale=0.8]
\node (dessin) at (0,1) {
\begin{tikzpicture}[scale=0.4]
\node[noeud,scale=0.4] (c) at (0,0) {};
\node[noeud,scale=0.4] (1) at (1,1) {};
\node[noeud,scale=0.4] (2) at (2,1) {};
\node[noeud,scale=0.4] (4) at (1,0) {};
\node[noeud,scale=0.4] (7) at (1,-1) {};
\node[noeud,scale=0.4] (8) at (2,-1) {};
\draw (c)--(1)--(2);
\draw (c)--(7)--(8);
\draw (c)--(4);
\end{tikzpicture}
};
\node[scale=0.8] (a) at (0,0) {$\grundy(\sstar{2,2,1})=0$};
\node[scale=0.8] (b) at (0,-0.5) {by Lemma~\ref{lem:124-small-ones}.};
\end{tikzpicture}
};
\node (s32) at (-6,-9) {
\begin{tikzpicture}[scale=0.8]
\node (dessin) at (0,1) {
\begin{tikzpicture}[scale=0.4]
\node[noeud,scale=0.4] (c) at (0,0) {};
\node[noeud,scale=0.4] (1) at (1,1) {};
\node[noeud,scale=0.4] (2) at (2,1) {};
\node[noeud,scale=0.4] (3) at (3,1) {};
\node[noeud,scale=0.4] (7) at (1,-1) {};
\node[noeud,scale=0.4] (8) at (2,-1) {};
\draw (c)--(1)--(3);
\draw (c)--(7)--(8);
\end{tikzpicture}
};
\node[scale=0.8] (a) at (0,0) {$\grundy(\sstar{3,2})=0$};
\node[scale=0.8] (b) at (0,-0.5) {by Lemma~\ref{lem:124-paths}.};
\end{tikzpicture}
};
\node (s22) at (3,-9) {
\begin{tikzpicture}[scale=0.8]
\node (dessin) at (0,1) {
\begin{tikzpicture}[scale=0.4]
\node[noeud,scale=0.4] (c) at (0,0) {};
\node[noeud,scale=0.4] (1) at (1,1) {};
\node[noeud,scale=0.4] (2) at (2,1) {};
\node[noeud,scale=0.4] (7) at (1,-1) {};
\node[noeud,scale=0.4] (8) at (2,-1) {};
\draw (c)--(1)--(2);
\draw (c)--(7)--(8);
\end{tikzpicture}
};
\node[scale=0.8] (a) at (0,0) {$\grundy(\sstar{2,2})=2$};
\node[scale=0.8] (b) at (0,-0.5) {by Lemma~\ref{lem:124-paths}.};
\end{tikzpicture}
};

\draw[->,thick] (s333) to (s332);
\draw[->,thick] (s333) to (s331);
\draw[->,thick] (s332) to (s322);
\draw[->,thick] (s332) to (s321);
\draw[->,thick] (s332) to (s33);
\draw[->,thick] (s331) to (s321);
\draw[->,thick] (s331) to (s33);
\draw[->,thick] (s331) to (s311);
\draw[->,thick] (s322) to (s321);
\draw[->,thick] (s322) to (s222);
\draw[->,thick] (s322) to (s221);
\draw[->,thick] (s322) to (s32);
\draw[->,thick] (s222) to (s221);
\draw[->,thick] (s222) to (s22);
\end{tikzpicture}
\caption{Analysis of $\sstar{3,3,3}$. The Grundy values are obtained using the mex rule and Lemmas~\ref{lem:124-paths} and~\ref{lem:124-small-ones}.}
\label{fig:124-S333}
\end{figure}

\begin{lemma}\label{lem:124-S22k}
For all $k\geq0$, we have $\grundy(\sstar{2,2,k})= |\sstar{2,2,k}|\bmod 3$, i.e. $\grundy(\sstar{2,2,k}) = \overline{201}$.
\end{lemma}

\begin{proof}
We proceed by induction on $k$. The cases $k\leq3$ are consequences of Lemma~\ref{lem:124-paths},~\ref{lem:124-small-ones} and~\ref{lem:124-S333}. When $k\geq4$, there are always exactly 5 moves available, leading to $\sstar{2,2,k-1}$, $\sstar{2,2,k-2}$, $\sstar{2,2,k-4}$, $\sstar{1,2,k}$, and $P_{k+3}$. We then conclude by induction and thanks to Lemma~\ref{lem:general}, since all these graphs (see Lemma~\ref{lem:124-paths} and~\ref{lem:124-S12k}) are such that $\grundy(G) = |G| \bmod 3$.
\end{proof}

\begin{lemma}\label{lem:124-S122k}
For all $k\geq0$, we have $\grundy(\sstar{1,2,2,k})= |\sstar{1,2,2,k}| \bmod 3$, i.e. $\grundy(\sstar{1,2,2,k}) = \overline{012}$.
\end{lemma}

\begin{proof}
We proceed by induction on $k$. The cases $k\leq1$ are consequences of Lemma~\ref{lem:124-small-ones}. For $k=2$, 3 moves are available from $\sstar{1,2,2,2}$, leading to $\sstar{1,1,2,2}$, $\sstar{1,2,2}$, and $\sstar{2,2,2}$, and we conclude thanks to Lemma~\ref{lem:124-small-ones} and~\ref{lem:124-S22k}. For $k=3$, 5 moves are available from $\sstar{1,2,2,3}$, leading to $\sstar{1,2,2,2}$, $\sstar{1,1,2,2}$, and $\sstar{1,1,2,3}$,  $\sstar{1,2,3}$, and  $\sstar{2,2,3}$, and we conclude thanks to Lemma~\ref{lem:124-small-ones} and~\ref{lem:124-S22k}. 
When $k\geq4$, there are always exactly 6 moves available, leading to $\sstar{1,2,2,k-1}$, $\sstar{1,2,2,k-2}$, $\sstar{1,2,2,k-4}$, $\sstar{1,1,2,k}$, $\sstar{1,2,k}$, and $\sstar{2,2,k}$. We conclude by induction, and thanks to Lemma~\ref{lem:124-S12k},~\ref{lem:124-S112k}, and~\ref{lem:124-S22k}, since we have:
\begin{itemize}
\item If $k \equiv 0 \bmod 3$, then $\grundy(\sstar{1,2,2,k})$ is computed as $\mex(2,1,2,2,1,2)=0$, and we are done.
\item If $k \equiv 1 \bmod 3$, then $\grundy(\sstar{1,2,2,k})$ is computed as $\mex(0,2,0,3,2,0)=1$, and we are done.
\item If $k \equiv 2 \bmod 3$, then $\grundy(\sstar{1,2,2,k})$ is computed as $\mex(1,0,1,1,0,1)=2$, and we are done.
\end{itemize}
\end{proof}

\begin{lemma}\label{lem:124-S1111k}
For all $k\geq0$, we have $\grundy(\sstar{1,1,1,1,k})= |\sstar{1,1,1,1,k}|+1 \bmod 3$, i.e. $\grundy(\sstar{1,1,1,1,k}) = \overline{012}$.
\end{lemma}

\begin{proof}
We proceed by induction on $k$. The case $k=0$ is a consequence of Lemma~\ref{lem:124-small-ones}. For $k=1$, 1 move is available from $\sstar{1,1,1,1,1}$, leading to $\sstar{1,1,1,1}$, and we conclude again thanks to Lemma~\ref{lem:124-small-ones}. For $k=2$, 3 moves are available from $\sstar{1,1,1,1,2}$, leading to $\sstar{1,1,1,1,1}$, $\sstar{1,1,1,1}$, and $\sstar{1,1,1,2}$, and we conclude thanks to Lemma~\ref{lem:124-small-ones} (and thanks to the case $k=0$ above). The case $k=3$ is not harder: 3 moves lead to $\sstar{1,1,1,1,2}$, $\sstar{1,1,1,1,1}$, and $\sstar{1,1,1,3}$, and we can conclude thanks to Lemma~\ref{lem:124-small-ones} and thanks to the cases $k\leq2$ above. When $k\geq4$, there are always exactly 4 moves available, leading to $\sstar{1,1,1,1,k-1}$, $\sstar{1,1,1,1,k-2}$, $\sstar{1,1,1,1,k-4}$, and $\sstar{1,1,1,k}$. We conclude by induction, and thanks to Lemma~\ref{lem:124-S111k}, since we have:
\begin{itemize}
\item If $k \equiv 0 \bmod 3$, then $\grundy(\sstar{1,1,1,1,k})$ is computed as $\mex(2,1,2,1)=0$, and we are done.
\item If $k \equiv 1 \bmod 3$, then $\grundy(\sstar{1,1,1,1,k})$ is computed as $\mex(0,2,0,0)=1$, and we are done.
\item If $k \equiv 2 \bmod 3$, then $\grundy(\sstar{1,1,1,1,k})$ is computed as $\mex(1,0,1,3)=2$, and we are done.
\end{itemize}
\end{proof}

Note that Lemma~\ref{lem:general} is not invoked in the two previous proofs. This is due to the cases $\sstar{1,1,2,k}$ and $\sstar{1,1,1,k}$, which appear in the general case $k\geq4$, but do not satisfy the properties $\grundy(G) = |G| \bmod 3$ and $\grundy(G) = |G|+1 \bmod 3$, respectively.

We are now ready to prove Theorem~\ref{thm:124}.

\begin{proof}[Proof of Theorem~\ref{thm:124}]
The proof is similar to that of Theorem~\ref{thm:123}. We denote A (for Alice) the first player, and B (for Bob) the second player, and we will show by induction that for any subdivided star $S=\sstar{\ell_1,\ldots,\ell_t,n}$ with $\ell_1,\ldots,\ell_t,n \geq 0$, then $S + S'$ is always a $\outcomeP$-position, with $S'$ defined as $\sstar{\ell_1,\ldots,\ell_t,n+3}$.

Let $b(S)$ be the number of nontrivial branches of $S$, that is to say $b(S) = |\{ i \mid \ell_i >0\}|$ if $n=0$ and $b(S) = |\{ i \mid \ell_i >0\}| + 1$ otherwise.

If $b(S)=0$ then we are trivially done thanks to Lemma~\ref{lem:124-paths} since $S=P_1$ and $S'=P_4$.

If $b(S)=1$, then we conclude again by Lemma~\ref{lem:124-paths} since $S$ and $S'$ are both paths.

If $b(S)\geq 2$, then several cases follow.

If A plays in $S'$ so as to leave $S + \sstar{\ell_1,\ldots,\ell_t,n+3-k}$, with $k\in\{1,2,4\}$, then B replies by taking $3-k$ vertices in $S'$ if $k\in\{1,2\}$, so as to leave $S + \sstar{\ell_1,\ldots,\ell_t,n} = S + S$, which is a $\outcomeP$-position, and if $k=4$ then B replies by taking 1 vertex in $S$, so as to leave $\sstar{\ell_1,\ldots,\ell_t,n-1} + \sstar{\ell_1,\ldots,\ell_t,n-1}$, which is a $\outcomeP$-position.

If A plays elsewhere in $S'$, then B can always copy A's move on $S$. For instance, if A removes $k\geq \ell_1$ vertices in $S'$, so as to leave $S \cup \sstar{\ell_1-k,\ldots,\ell_t,n+3}$, then B removes $k$ vertices in $S$ so as to leave $\sstar{\ell_1-k,\ldots,\ell_t,n} \cup \sstar{\ell_1-k,\ldots,\ell_t,n+3}$. When B replies by mimicking A's move on $S$, B manages to leave $T + T'$ to A, with $|T| < |S|$, and we thus can conclude by induction that this is a $\outcomeP$-position.

If A plays in $S$, then we have to be careful. If A's move does not involve taking the center of $S$, then we are done, since B is again always able to copy A's move in $S'$, so as to leave $T + T'$ to A, with $|T| < |S|$, and we again conclude by induction that this is a $\outcomeP$-position.

The tricky cases are the ones where A takes the center of $S$. Indeed, in this case, B can not always simply copy A's move in $S'$, and we have to carry a more detailed analysis. For instance, if $S=\star{1,1,1,2}$ and $S'=\sstar{1,1,1,2,3}$, and if A removes 4 vertices so as to leave $P_2 + \sstar{1,1,1,2,3}$, then B can not mimic A's move in $S'$ since it would disconnect $\sstar{1,1,1,2,3}$, so as to leave $P_2 + P_2 + P_3$. In other words, A's move copied in $S'$ is not always a legal move.

Let us assume now that A's move in $S$ can not be copied by player B in $S'$. In this case, A was able to remove the centre of $S$, and this implies that $S$ has at most 4 nontrivial branches, and that $S'$ has at most one more nontrivial branch than $S$. Note also that A removed strictly more than 1 vertex to $S$, and exactly $b(S)-1$ branches of $S$ (if A empties $S$ then B can trivially copy A's move on $S'$ so as to leave $P_3$). Three cases and several subcases follow. We will denote by $k$ the number of vertices removed by A. 

\textbf{Case (i): $b(S)=2$.}
Without loss of generality let us assume $S=\sstar{\ell_1,\ell_2}$ with $1\leq\ell_1 \leq \ell_2$. As in the case of Theorem~\ref{thm:123}, if there is a long enough branch in $S$ (w.r.t the number of vertices removed by A), then we will proceed by a mirror argument, since $S$ is actually a path $P_{\ell_1 + \ell_2 +1}$. More precisely, if $k\leq\ell_2$, then we may consider A actually played so as to leave $\sstar{\ell_1,\ell_2-k} + S'$, and in this case player B can copy A's move on $S'$, and we are done.

Let us assume now that $k>\ell_2$. If $k=2$, then $S = \sstar{1,1}$, and we may have either $S'=\sstar{1,4}=P_6$ or $S'=\sstar{1,1,3}$. If $S'=\sstar{1,4}=P_6$, then A may remove 2 vertices from $P_6$ so as to leave $P_1 + P_4$, which is a $\outcomeP$-position. If $S'=\sstar{1,1,3}$, then A may remove 2 vertices from $\sstar{1,1,3}$ so as to leave $P_1 + \sstar{1,1,1}$, which is a $\outcomeP$-position since we know from Lemma~\ref{lem:124-small-ones} that $\grundy(\sstar{1,1,1}) = 1 = \grundy(P_1)$.

If $k=4$, then there are several subcases.

\begin{itemize}
\item If $S=\sstar{1,2}$, then A leaves $\emptyset +S'=S'$ and we may have either $S'=P_7$ or $S'=\sstar{1,2,3}$. If $S'=P_7$, then A may remove 1 vertex from $P_7$ so as to leave $P_6$, which is a $\outcomeP$-position. If $S'=\sstar{1,2,3}$, then A may remove 1 vertex from $\sstar{1,2,3}$ so as to leave $P_6$, which is a $\outcomeP$-position.
\item If $S=\sstar{1,3}$, then A leaves $P_1 +S'$ and we may have either $S'=P_8$ or $S'=\sstar{1,3,3}$. If $S'=P_8$, then A may remove 1 vertex from $P_8$ so as to leave $P_1 + P_7$, which is a $\outcomeP$-position. If $S'=\sstar{1,3,3}$, then A may remove 1 vertex from $\sstar{1,3,3}$ so as to leave $P_1 + \sstar{1,2,3}$, which is a $\outcomeP$-position thanks to Lemma~\ref{lem:124-small-ones}.
\item If $S=\sstar{2,2}$, then A leaves $P_1 +S'$ and we may have either $S'=P_8$ or $S'=\sstar{2,2,3}$. If $S'=P_8$, then A may remove 1 vertex from $P_8$ so as to leave $P_1 + P_7$, which is a $\outcomeP$-position. If $S'=\sstar{2,2,3}$, then A may remove 1 vertex from $\sstar{2,2,3}$ so as to leave $P_1 + \sstar{2,2,2}$, which is a $\outcomeP$-position thanks to Lemma~\ref{lem:124-S333}.
\item If $S=\sstar{2,3}$, then A leaves $P_2 +S'$ and we may have either $S'=P_9$ or $S'=\sstar{2,3,3}$. If $S'=P_9$, then A may remove 1 vertex from $P_9$ so as to leave $P_2 + P_8$, which is a $\outcomeP$-position. If $S'=\sstar{2,3,3}$, then A may remove 1 vertex from $\sstar{2,3,3}$ so as to leave $P_2 + \sstar{2,2,3}$, which is a $\outcomeP$-position thanks to Lemma~\ref{lem:124-S333}.
\item If $S=\sstar{3,3}$, then A leaves $P_3 +S'$ and we may have either $S'=P_{10}$ or $S'=\sstar{3,3,3}$. If $S'=P_{10}$, then A may remove 1 vertex from $P_{10}$ so as to leave $P_3 + P_9$, which is a $\outcomeP$-position. If $S'=\sstar{3,3,3}$, then A may remove 1 vertex from $\sstar{3,3,3}$ so as to leave $P_3 + \sstar{2,3,3}$, which is a $\outcomeP$-position thanks to Lemma~\ref{lem:124-S333}.
\end{itemize}

\textbf{Case (ii): $b(S)=3$.}
Let us denote $S=\sstar{\ell_1,\ell_2,\ell_3}$. In this case, we must have $k=4$. Two subcases follow.

\begin{itemize}
\item If $S = \sstar{1,1,\ell_3}$, then A leaves $P_{\ell_3-1} + S'$, and we may have $S'=\sstar{1,1,3,\ell_3}$, $S'=\sstar{1,1,\ell_3+3}$, or $S'=\sstar{1,4,\ell_3}$.
\begin{itemize}
\item If $S'=\sstar{1,1,3,\ell_3}$, then B's answer depends on $\ell_3 \bmod 3$. If $\ell_3 \equiv 0\bmod 3$, B may reply by taking 1 vertex in $S'$, so as to leave $P_{\ell_3-1} + \sstar{1,1,2,\ell_3}$. Indeed, from Lemma~\ref{lem:124-S112k}, we know that $\grundy(\sstar{1,1,2,\ell_3})=2=\grundy(P_{\ell_3-1})$, thus $P_{\ell_3-1} + \sstar{1,1,2,\ell_3}$ is a $\outcomeP$-position. If $\ell_3 \equiv 1\bmod 3$, B may reply by taking 2 vertices in $S'$, so as to leave $P_{\ell_3-1} + \sstar{1,1,1,\ell_3}$. Indeed, from Lemma~\ref{lem:124-S111k}, we know that $\grundy(\sstar{1,1,1,\ell_3})=0=\grundy(P_{\ell_3-1})$, thus $P_{\ell_3-1} + \sstar{1,1,1,\ell_3}$ is a $\outcomeP$-position. If $\ell_3 \equiv 2\bmod 3$, B may reply by taking 1 vertex in $S'$, so as to leave $P_{\ell_3-1} + \sstar{1,1,2,\ell_3}$. Indeed, from Lemma~\ref{lem:124-S112k}, we know that $\grundy(\sstar{1,1,2,\ell_3})=1=\grundy(P_{\ell_3-1})$, thus $P_{\ell_3-1} + \sstar{1,1,2,\ell_3}$ is a $\outcomeP$-position.
\item If $S'=\sstar{1,1,\ell_3+3}$, then B removes 4 vertices from $S'$, so aas to leave $P_{\ell_3-1}+P_{\ell_3-1}$, which is a $\outcomeP$-position.
\item If $S'=\sstar{1,4,\ell_3}$, then B removes 4 vertices from $S'$, so aas to leave $P_{\ell_3-1}+P_{\ell_3+2}$, which is a $\outcomeP$-position.
\end{itemize} 
\item If $S = \sstar{1,2,\ell_3}$, then A leaves $P_{\ell_3} + S'$, and we may have $S'=\sstar{1,2,3,\ell_3}$, $S'=\sstar{1,2,\ell_3+3}$, $S'=\sstar{1,5,\ell_3}$, or $S'=\sstar{2,4,\ell_3}$. If $S'=\sstar{1,2,3,\ell_3}$, B can reply by taking 1 vertex in $S'$, so as to leave $P_{\ell_3} + \sstar{1,2,2,\ell_3}$, which is a $\outcomeP$-position. Indeed, from Lemma~\ref{lem:124-S122k}, we know that $\grundy(\sstar{1,1,2,\ell_3})=\grundy(P_{\ell_3})= \ell_3 \bmod 3$. If $S'=\sstar{1,2,\ell_3+3}$, then B may reply by removing 1 vertex from $S'$, so as to leave $P_{\ell_3} + P_{\ell_3+3}$, which is a $\outcomeP$-position. If $S'=\sstar{1,5,\ell_3}$, then B may reply by removing 1 vertex from $S'$, so as to leave $P_{\ell_3} + P_{\ell_3+6}$, which is a $\outcomeP$-position. Finally, if $S'=\sstar{2,4,\ell_3}$, then B may reply by removing 4 vertices from $S'$, so as to leave $P_{\ell_3} + P_{\ell_3+3}$, which is a $\outcomeP$-position. 
\end{itemize}

\textbf{Case (iii): $b(S)=4$.}
In this case, we must have $k=4$ and $S=\sstar{1,1,1,\ell_4}$, and A leaves $P_{\ell_4} + S'$, with $S'=\sstar{1,1,1,3,\ell_4}$, $S'=\sstar{1,1,1,\ell_4+3}$, or $S'=\sstar{1,1,4,\ell_4}$. If $S=\sstar{1,1,1,3,\ell_4}$, then a good answer for B is then to remove 2 vertices from $S'$, so as to leave $P_{\ell_4} + \sstar{1,1,1,1,\ell_4}$, which is a $\outcomeP$-position, since we know from Lemma~\ref{lem:124-S1111k} that $\grundy(\sstar{1,1,1,1,\ell_4})=\grundy(P_{\ell_4})= \ell_4 \bmod 3$. If $S=\sstar{1,1,1,\ell_4+3}$, then B may remove 4 vertices from $S'$, so as to leave $P_{\ell_4} + P_{\ell_4+3}$, which is a $\outcomeP$-position. If $S=\sstar{1,1,4,\ell_4}$, then B may remove 4 vertices from $S'$, so as to leave $P_{\ell_4} + \sstar{1,1,\ell_4}$, which is a $\outcomeP$-position, since we know from Lemma~\ref{lem:124-S11k} that $\grundy(\sstar{1,1,\ell_4})=P_{\ell_4}=\ell_4 \bmod 3$.
\end{proof}

\section{Conclusion}

In this paper, we derived results for connected subtraction games in graphs. Our main results deal with graphs to which we append paths, for which we were able to derive a very general result showing that there is always ultimate periodicity, in the sense that, for any finite set $L$, there exists two integers $k_0$ and $T$ such that
$$\grundy_L(\app{G}{u}{k+T}) = \grundy_L(\app{G}{u}{k}),$$
whenever $k\geq k_0$ (Theorem~\ref{thm:general}). Is it true that the period $T$ is the same as the period of the game played on paths?

We then managed to prove pure periodicity results in some particular cases. We showed that, for some classes of graphs and sets $L$, equation $\grundy_L(\app{G}{u}{k+T}) = \grundy_L(\app{G}{u}{k})$ holds starting from $n_0=0$. This is the case, for instance, for simple stars when playing \csg{\{1,\ldots,N\}} (Proposition~\ref{prop:Stk_general}), for subdivided stars when playing \csg{\{1,2,3\}} or \csg{\{1,2,4\}} (Theorems~\ref{thm:123} and~\ref{thm:124}).

However, even in graphs as structurally simple as subdivided stars, pure periodicity is not the general setting, even if the game is purely periodic when played in paths. We proved, for instance, in Theorem~\ref{thm:S1lk} that we only have ultimate periodicity in subdivided stars of type $\sstar{1,k,l}$ for the game \csg{\{1,\ldots,N\}}. For instance, as a corollary of Theorem~\ref{thm:S1lk}, we have $\grundy(\sstar{1,2,8})=10$ whereas $\grundy(\sstar{1,2+9k,8})=6$ for all $k\geq1$, for the game \csg{\{1,\ldots,8\}}. Note that this game is purely periodic when played in paths, and that subdivided stars of the form $\sstar{1,k,l}$ can be seen as paths to which we attached a vertex. Hence there is a complexity gap in the computation of the Grundy values when considering subdivided stars instead of paths. This complexity gap between paths and stars is also illustrated by Observation~\ref{obs:plus_M}, which states that, for subdivided stars -- and even for stars -- the Grundy sequence of \csg{I_N \cup \{M\}} may be different from the one of \csg{I_N}, for $M \not\equiv 0 \bmod{N+1}$.



\end{document}